\documentclass[a4paper,10pt]{amsart}
\usepackage{amssymb,amsmath,amsfonts,amsthm}
\usepackage[dvips]{graphicx}
\usepackage{psfrag}

\theoremstyle{plain}
\newtheorem{main}{Theorem}

\newtheorem{maincor}[main]{Corollary}
\newtheorem{theorem}{Theorem}[section]
\newtheorem{lemma}[theorem]{Lemma}
\newtheorem{proposition}[theorem]{Proposition}
\newtheorem{corollary}[theorem]{Corollary}
\theoremstyle{remark}
\newtheorem{remark}[theorem]{Remark}
\newtheorem{definition}[theorem]{Definition}

\newtheorem{conjecture}[theorem]{Conjecture}

\newcommand{\Leb}{\operatorname{vol}}

\newcommand{\C}{\operatorname{C}}
\newcommand{\Sing}{\operatorname{Sing}}

\newcommand{\Card}{\operatorname{Card}}
\newcommand{\Jac}{\operatorname{Jac}}

               \def\cal{\mathcal}
           \def\ea{\end{array}}
          \def\ec{\end{center}}
     \def\ed{\end{description}}
        \def\ee{\end{equation}}
       \def\eea{\end{eqnarray}}
     \def\eeaa{\end{eqnarray*}}
 \def\et{\end{thebibliography}}
\def\bib{\bibitem}

\def\R{{\cal R}}

\def\Orb{{\rm Orb}}
\def\Diff{{\rm Diff}}

\def\Sing{{\rm Sing}}

\def\supp{\operatorname{supp}}

\def\cG{{\mathcal G}}
\def\cA{{\mathcal A}}
\def\cD{{\mathcal D}}
\def\cC{{\mathcal C}}
\def\cO{{\mathcal O}}
\def\cI{{\mathcal I}}

\def\cU{{\mathcal U}}

\def\cR{{\mathcal R}}

\def\cF{{\mathcal F}}
\def\cM{{\mathcal M}}
\def\cN{{\mathcal N}}
\def\cP{{\mathcal P}}
\def\cT{{\mathcal T}}
\def\cR{{\mathcal R}}

\def\length{\operatorname{length}}

\def\vep{\varepsilon}
\def\ln{\operatorname{ln}}

\def\sp{\operatorname{sp}}

\title{Topological entropy of Lorenz-like flows}

\author{Jiagang Yang}
\date{\today}
\thanks{J.Y. is partially supported by CNPq, FAPERJ, and PRONEX.}
\address{Departamento de Geometria, Instituto de Matem\'atica e Estat\'istica, Universidade
Federal Fluminense, Niter\'oi, Brazil}
\email{yangjg\@@impa.br}

\begin{document}

\begin{abstract} We use entropy theory as a new tool for studying Lorenz-like classes of flows in any
dimension. More precisely,
we show that every Lorenz-like class is entropy expansive, and has positive 
entropy which varies continuously with vector fields. We deduce that every such class 
contains a transverse homoclinic orbit and, generically, is an attractor.

\end{abstract}

\maketitle

\setcounter{tocdepth}{1} \tableofcontents


\section{Introduction}

The study of Lorenz attractors, comes from the famous article of Lorenz~\cite{L63}, 
where he did numerical investigation of the following system:

\begin{equation*}
\left\{
\begin{aligned}
\dot{x}&=-\sigma x+\sigma y \;\;\;\;\;\;\;\;\; \sigma=10       \\
\dot{y}&=r\-y-xz\;\;\;\;\;\;\;\;\;\;\;\;    r=28\\
\dot{z}&=-bz+xy\;\;\;\;\;\;\;\;\;  b=8/3\\
\end{aligned}
\right.
\end{equation*}
Lorenz observed that, the solutions of the above simple equations, are sensitive with
respect to the initial conditions. Sometime later, a 
geometric Lorenz model for this attractor was proposed in~\cite{Gu76,ABS77,GW79, Wi79}, which 
provides a classical non-uniformly hyperbolic, but robust transitive example for time continuous 
systems. Finally it was proved much later by~\cite{Tu99, Tu01}, with the help of computer that, 
the Lorenz attractor does exist and is exactly a kind of the geometric model. 
The history of this story can be found in~\cite{V}. 

Inspired by the proposition of the Lorenz-like class, 
a theory for 3-dimensional robustly transitive strange attractor was built in~\cite{MPP98, MPP04}, 
which showed that, any robust attractor of a 3-dimensional flow which contains both singularity and regular 
orbits must be singular-hyperbolic, that is, it admits a dominated splitting $E^s\oplus F^{cu}$ on 
the tangent bundle into a 1-dimensional uniformly contracting sub-bundle and a 2-dimensional volume-expanding 
sub-bundle. The notion of Lorenz-like class, is a generalization of the above constructions and results.

\begin{definition}\label{d.lorenzlike}

A compact invariant set $\Lambda$ of a $\C^1$ flow $\phi_t$ is a {\it Lorenz-like class} if it contains both
singularity and regular points, and satisfies the following conditions:

\begin{itemize}
\item[(a)] $\Lambda$ is a chain recurrent class;

\item[(b)] $\Lambda$ is {\it Lyapunov stable}, i.e., there is a sequence of compact neighborhoods $\{U_i\}$ 
such that: 
\begin{itemize}
\item $U_1 \supset U_2\supset\dots$ and $\cap_{i\geq 1} U_i=\Lambda$;
\item for each $i\geq 1$, $\phi_t(U_{i+1})\subset U_i$ for any $t\geq 0$.
\end{itemize}

\item[(c)] $\Lambda$ is sectional-hyperbolic, i.e., $\Lambda$ admits a dominated splitting $E^s\oplus F^{cu}$
such that the bundle $E^s$ is uniformly contracting, and the bundle $F^{cu}$ is sectional-expanding, which means,
there are $K,\lambda>0$ such that for any subspace $V_x\subset F^{cu}(x)$ with $x\in \Lambda$ and 
$dim(V_x)\geq 2$, we have:
$$|\det(DX_t(x)|_{V_x})|\geq K\exp^{\lambda t}\;\; \text{for all}\;\; t>0.$$

\end{itemize}
\end{definition}

Note that we \textbf{neither} assume the singularities to be hyperbolic \textbf{nor} the class to be isolated. An example 
for higher dimensional Lorenz-like class (which is, in fact, an attractor), was constructed in~\cite{BPVbis} 
with $dim(F^{cu})>2$. More recently,~\cite{SGW} proved that, for generic star flows, every non-trivial 
Lyapunov stable chain recurrent class is Lorenz-like, where a $\C^1$ flow is a {\it star flow} if for any flow 
nearby, its critical elements are all hyperbolic (see also~\cite{LGW},~\cite{ZGW}). 

The results on 3-dimensional Lorenz attractors are quite fruitful, e.g., see~\cite{MPP98, MPP04, APPV09, AP, AP07}. In fact, 
as in the construction of geometric Lorenz model, one can take 2-dimensional sections near to every singularity 
and analyze the Poincar\'e return maps. The corresponding Poincar\'e return maps are partially hyperbolic, the quotient 
along the stable lamination induces a family of 1-dimensional Lorenz maps (see~\cite{APPV09}). 
By this argument, several important properties were proved, for example:
every 3-dimensional singular hyperbolic attractor is a homoclinic class (~\cite{AP}) and is expansive 
(\cite{APPV09}). This argument is powerful, but is hard to apply on general Lorenz-like classes. One reason is 
that Lorenz-like classes may not be isolated, it will be difficult to choose sections. Another reason is that 
the corresponding higher dimensional Lorenz map is not well-studied. The non-hyperbolicity of the singularities also 
brings difficulty on the analysis of the orbits close to the singularities. In this paper, we will provide a new 
method to study Lorenz-like classes, which is based on the new development of entropy theory (~\cite{LVY}).

\begin{main}\label{main.A}

Let $\Lambda$ be a Lorenz-like class of a $\C^1$ flow $\phi_t$, then $\phi_t$ has positive topological entropy on $\Lambda$
and is entropy expansive in a neighborhood of $\Lambda$.

\end{main}

The positive topological entropy of Lorenz-like class implies the existence of non-trivial homoclinic class.

\begin{maincor}\label{main.B}

Let $\Lambda$ be a Lorenz-like class of a $\C^1$ flow $\phi_t$, then $\Lambda$ contains a non-trivial homoclinic class.
Moreover, if $\phi_t$ belongs to a $\C^1$ residual subset of flows, then $\Lambda$ is an attractor.

\end{maincor}

Combining the result of~\cite{SGW}, one can get the following description for star flows:

\begin{maincor}\label{main.C}

There is a $\C^1$ residual subset $\R$ of star flows, such that for any star flow $\phi_t\in \cR$, $\phi_t$ has finitely many
attractors, whose basins cover an open and dense subset of the ambient manifold.

\end{maincor}

We also obtain regularity of the topological entropy for Lorenz like classes which are quasi attractors, where 
an invariant compact set $\Lambda$ is {\it quasi attractor} if for any neighborhood $V$ of $\Lambda$, 
there is compact set $U\subset V$ such that $\phi_t(U)\subset int(U)$ for any $t>0$. 
The continuity of topological entropy for 1-dimensional Lorenz maps was proved in~\cite{Mal82, Shl88}.

\begin{main}\label{main.D}

Let $\Lambda$ be a Lorenz-like class of a $\C^1$ flow $\phi_t$ with all singularities hyperbolic. Suppose
$\Lambda$ is a quasi attractor, then 
there is an open neighborhood $U$ of $\Lambda$ and a neighborhood $\cU$ of $\phi_t$ such that the topological entropy
$h_{top}(. |_{U})$ varies continuously respect to the flows in $\cU$.

\end{main}

\begin{remark}
Let $U$ be an attracting compact set, i.e., $\phi_t(U)\subset int(U)$, and $\Lambda$ the maximal invarint set of $\phi_t|_U$. Suppose
$\Lambda$ is sectional hyperbolic and not necessary to be chain recurrent. Then it is easy to check that our proofs for Theorems~\ref{main.A} and~\ref{main.D} still work.
\end{remark}

The first part of Theorem~\ref{main.A}, as well as Corollary~\ref{main.B}, exist an alternative proof, which 
was announced at the International Conference on Dynamical Systems held at IMPA, Rio de Janeiro in November 2013. A video of
this lecture is available at the conference website~\cite{video}. Similar methods and results to that lecture
appeared in a preprint circulated by A. Arbieto, C.A. Morales, A.M. Lopez B in August 2014.

Now let us explain quickly how the entropy theory is used in the proof:
We study the time-one map of the flow in a neighborhood of Lorenz-like class. The advantage to take the time-one 
diffeomorphism is because, we may make use of `fake foliations', which comes from the dominated splitting of 
the time-one map, and is in general not preserved by the flow.
Making use of the fake foliation and a new criterion for entropy expansiveness of~\cite{LVY}, we prove that
the time-one map is entropy expansive in a small neighborhood of this Lorenz-like class. In particular,
the metric entropy is upper semi-continuous in this neighborhood.

Then we show a lower bound of topological entropy for diffeomorphisms admitting a dominated
splitting, which enables us to prove that, the topological entropy for flow in any small neighborhood of this Lorenz
like class is uniformly bounded from zero. Taking a sequence of neighborhoods which converge to the Lorenz-like
class, and combining the upper semi-continuity of the entropy, we prove that the topological entropy on this Lorenz-like class is
always positive.

This paper is organized as follows. In Section~\ref{s.entropy} we prove Theorem A. And
in Section~\ref{s.lowercontinuiation}, we prove the lower semi-continuation of topological entropy and
finish the proof of Corollaries~\ref{main.B},~\ref{main.C} and Theorem~\ref{main.D}.

\section{Preliminary}
Throughout this article, $X$ denotes a $\C^1$ vector field on a $d$-dimensional 
closed manifold $M^d$, $\Sing(X)$ the singularities of $X$, 
$\phi_t$ the flow generated by $X$, and $f$ the corresponding time-one map, i.e., $f=\phi_1$. We also denote the 
corresponding tangent flow by $\Phi_t=d\phi_t: TM^d\rightarrow TM^d$.

\subsection{Dominated splitting\label{ss.dominated splitting}}

Throughout this subsection, $g\in \Diff^1(M)$ is a diffeomorphism which admits a {\it dominated splitting} 
$E\oplus F$ on the tangent space, i.e., there exists $L\in \mathbb{N}$ such that for every $x \in M$, and every pair of non-zero 
vectors $u\in E(x)$ and $v\in F(x)$, one has
$$\frac{\|Dg_x^L(u)\|}{\|u\|}\leq \frac{1}{2}\frac{\|Dg^L_x(v)\|}{\|v\|}.$$

For $a>0$ and $x\in M$, we define $(a, F)$-cone on the tangent space $T_xM$:
$$\cC_a(F_x)=\{v; v=0 \; \text{or}\; v=v_E+v_F\;\; \text{where}\;\; v_E\in E;\; v_F\in F\;\; \text{and}\;\; \frac{\|v_E\|}{\|v_F\|}<a.\}$$  
When $a$ is sufficiently small, the cone fields $\cC_a(F_x)$, $x\in M$ is forward 
invariant, i.e., there is $\lambda<1$ such that for any $x\in M$, $Dg_x(\cC_a(F_x))\subset \cC_{\lambda a}(g(x))$.
In a similar way, we may define the $(a,E)$-cone $\cC_a(E_x)$, which is backward invariant.
When no confusing is caused, we call the two families of cones by $F$ cone and $E$ cone.

The images of the cones under the exponential map are also forward or backward invariant. The invariance 
is induced from the invariance of the $E$ and $F$ cones on the tangent space. 
More precisely: Fix $\vep_0>0$ small enough, such that the exponential map is well defined 
on the $\vep_0$ ball in the tangent space. The image of the exponential map restricted 
on the set $B_{\vep_0}(0)\cap \cC_a(F_x)\subset T_xM$ defines a {\it local $F$ cone} in 
$B_{\vep_0}(x)$, which we denote by $\C_a(F_x)$. Then for any $x\in M$, we have:
$$g\big(\C_a(F_x)\cap B_{\frac{\vep_0}{\|f\|_{\C^1}}}(x)\big)\subset \C_{\lambda a}(F_{f(x)}).$$
In the same way we define $\C_a(E_x)$.

\begin{definition}\label{l.integrable}
Let $D$ be a $\C^1$ disk with dimension $\dim(E)$. We say $D$ is:
\begin{itemize}
\item {\it tangent to} {\it $F$ cone} if for any $x\in D$, $T_xD\subset \cC_a(F_x)$;
\item {\it tangent to} {\bf local} {\it $F$ cone at $x$} if $D\subset \C_a(F_x)$; 
\item {\it tangent to} {\bf local} {\it $F$ cone} if for any $y\in D$ we have $D\subset \C_a(F_y)$.
\end{itemize}
\end{definition}

$D$ is tangent to local $F$ cone implies that it is tangent to $F$ cone. Conversely, if $D$ 
is tangent to $F$ cone, then it can be divided into finitely many sub-disks, each sub-disk 
is tangent to local $F$ cone. 

\begin{remark}\label{r.transverse}

Topologically, the local cones $\C_a(E_x)$ and $\C_a(F_x)$ for $x\in M$ are transverse to 
each other, that is, $\C_a(E_x)\cap \C_a(F_x)=\{x\}$.

\end{remark}

\begin{remark}\label{r.localtanency}

Suppose $D$ is a disk with dimension $\dim(F)$ and transverse to $E$ bundle, then there is $n>0$ 
sufficiently large, such that $g^n(D)$ is tangent to $F$ cone. Hence, it can be divide into finitely 
connected pieces $g^n(D)=\cup_{i=1}^l D_i$, such that each piece $D_i$ is tangent to local $F$ cone.
\end{remark}

\begin{lemma}\label{l.onesteplocalvolume}

There is $K_0>0$ such that for any $x\in M$ and any disk $D\subset B_{\vep_0}(x)$ which is tangent to the local 
$F$ cone, $\Leb(D)\leq K_0$.

\end{lemma}

\begin{proof}

Take a smooth foliation $\cF$ in $B_{\vep_0}(x)$ such that every leaf has dimension $dim(E)$ and is 
tangent to $E$ cone. Fix a disk $D_x\ni x$ which is tangent to $F$ cone, for example, we may choose 
the image of $exp_x(B_{\vep_0}(x)\cap F_x)$. By Remark~\ref{r.localtanency}, each leaf of $\cF$ 
intersects $D$ with at most one point. Hence the foliation $\cF$ induces a holonomy map 
$H: D\rightarrow D_x$. Because both $D$ and $D_x$ are tangent to $F$ cone, the Jacobian of $H$ is 
uniformly bounded from above and below. Observe that the volume of $D_x$ is uniformly bounded from 
above, the volume of $D_x$ is bounded by a uniform constant $K_0>0$.

\end{proof}

By the invariance of local cones and induction, it is easy to prove the following lemma:

\begin{lemma}\label{l.localtangency}

For any $x\in M$, $\vep<\frac{\vep_0}{\|g\|_{\C^1}}$ and $n>0$, if $D\subset B_{\vep}(x)$ is tangent 
to local $F$ cone (at $x$) and $g^i(D)\subset B_{\vep}(g^i(x))$ for every $0\leq i \leq n$, 
then $g^n(D)$ is tangent to local $F$ cone (at $g^n(x)$).

\end{lemma}

\begin{definition}\label{d.Fexpansion}

Let $D$ be a disk tangent to the $F$ cone, then the {\it volume expansion of $D$}, $v_F(D)$, is defined by $\limsup_{n}\frac{1}{n}\log(\Leb(g^n(D)))$. The {\it volume expansion $v_F$ of bundle $F$}, is defined by:
$$v_F=\sup\{v_F(D);\;\; D\;\; \text{is tangent to the}\;\; F\;\; \text{cone}\}.$$ 

\end{definition}

\subsection{Entropy}
Let $g : M \rightarrow M$ be a continuous map on a compact metric space $M$, 
$K$ be a subset of $M$ not necessary to be invariant. For each $\vep > 0$ and 
$n \geq 1$, we consider the the dynamical ball of radius $\vep > 0$ and length 
$n$ around $x \in M$:
$$
B_n(x,\vep)=\{y\in M ; d(g^j(x),g^j(y)) \leq \vep \; \text{for every}\; 0 \leq j < n \}.
$$

A set $E \subset M$ is {\it ($n,\vep$)-spanning} for $K$ if for any $x\in K$ 
there is $y\in E$ such that $d(g^i(x),g^i(y))\leq \vep$ for all $0 \leq i < n$. 
In other words, the dynamical balls $B_n(y,\vep), y\in E$ cover $K$. Let 
$r_n(K,\vep)$ denote the smallest cardinality of any ($n,\vep$)-spanning set, and
$$
r(K,\vep) = \limsup_{n\rightarrow +\infty}\frac{1}{n}\log r_n(K,\vep).
$$

The topological entropy of $g$ on $K$ is defined by $h_{top}(g, K) = \lim_{\vep\rightarrow 0} 
r(K, \vep)$. And the topological entropy of $f$ is defined as $h_{top}(g)=h_{top}(g,M)$. 
The topological entropy of a continuous flow equals to the topological entropy of its 
time-one map.

Let $\mu$ be an invariant measure and $\cA$ a finite partition, the metric entropy of 
$\mu$ corresponding to the partition $\cA$ is defined as
$$
h_\mu(\cA)=\lim_n -\sum\mu(B_{n,i})\log \mu(B_{n,i}),
$$
where $B_{n,i}$ are elements of partition
$$
\cA_0^{n-1}=\cA \vee g^{-1}\cA \vee \dots \vee g^{1-n}\cA.
$$
The metric entropy of an invariant measure $\mu$ is defined as
$$
h_\mu=\sup_{\cA\; \text{finite partition}}\{h_\mu(\cA)\}.
$$

By the variational principle, $h_{top}(g)=\sup_{\mu\in \cM_{inv}(g)}h_\mu$, where 
$\cM_{inv}(g)$ denotes the space of invariant probability of $g$. Metric entropy is not 
necessary to be upper semi-continuous depending on the invariant measures, and the topological 
entropy may not be achieved by any metric entropy, e.g., see~\cite{New89,DN05}.

For each $x \in M$ and $\vep > 0$, let $B_\infty(x,\vep)=\{y: d(g^n(x),g^n(y))\leq\vep$ 
for $n\in \mathbb{Z}\}$. $g$ is ($\vep$-){\it entropy expansive} if
$$
\sup_{x\in M}h_{top}(g,B_\infty(x,\vep)) = 0.
$$

For $\vep>0$, we say $g$ is $\vep$-{\it almost entropy expansive} if for any invariant 
ergodic measure $\mu$ of $g$ and $\mu$ almost every point $x\in M$, we have 
$ h_{top}(g, B_\infty(x,\vep))=0.$

A new criterion of entropy expansiveness was given in \cite{LVY}[Proposition 2.4]:

\begin{lemma}\label{l.entropyexpansive}

$g$ is $\vep$-almost entropy expansive if and only if it is $\vep$-entropy expansive.

\end{lemma}

The positivity of topological entropy of Theorem~\ref{main.A} depends on the
following theorem, whose proof, as well as ones for the Corollaries~\ref{c.Fentropyconjecture},~\ref{c.transtive}, will be
given in Section~\ref{s.positiveentropy}.

\begin{theorem}\label{t.Flowboundary}

Suppose $g$ is a diffeomorphism which admits a dominated splitting $E\oplus F$. Then $h_{top}(g)\geq v_F$.

\end{theorem}

Shub~\cite{Shu74} has conjectured an estimation of lower bound of topological entropy: 
the topological entropy is bounded from below by the spectral radius in homology 
(see also Shub, Sullivan~\cite{SS75}). 

Let $d = dim M$ and $g_{*,k} : H_k(M,\mathbb{R})\rightarrow 
H_k(M,\mathbb{R}), 0 \leq k \leq d$ be the action induced by $g$ on the real homology groups of $M$. Let
$$\sp(g_*) = \max_{0\leq k\leq d} \sp(g_{*,k}),$$
where $\sp(g_{*,k})$ denotes the spectral radius of $g_{*,k}$. The Shub entropy conjecture states that:

\begin{conjecture}\label{c.entropy conjecture}

For every $g\in \Diff^1(M)$, $h_{top}(g)\geq \log \sp(g_*) $.

\end{conjecture}

This conjecture has been proved for $C^\infty$ maps in~\cite{Yom87} and for diffeomorphims away from 
tangencies in~\cite{LVY}. The history and more references on the Shub entropy conjecture can be found in \cite{LVY}. 
As an immediate corollary of Theorem~\ref{t.Flowboundary}, we obtain a partial answer to the 
Shub entropy conjecture.

\begin{corollary}\label{c.Fentropyconjecture}

Suppose $g$ is a diffeomorphism which admits a dominated splitting $E\oplus F$, then 
\begin{equation}\label{eq.1}
h_{top}(g)\geq \log \sp(g_{*,dim(F)}).	
\end{equation}

\end{corollary}

Theorem~\ref{t.Flowboundary} and Corollary~\ref{c.Fentropyconjecture} were first stated in~\cite{SX10}, but their proof only works 
when the bundle $F$ is uniformly expanding (see the estimation of Bowen ball
in the proof of~\cite{SX10}[Proposition 2]), we will give more explanation in Remark~\ref{r.SX}. 

It was shown in~\cite{BDP03} that every $\C^1$ robustly transitive diffeomorphism admits a dominated splitting,
and the extreme bundles are volume hyperbolic. As a corollary of Theorem~\ref{t.Flowboundary}, we prove that:

\begin{corollary}\label{c.transtive}

Suppose $g$ is a robust transitive diffeomorphism. Then its topological entropy is positive.

\end{corollary}

New results on entropy expansiveness for partially hyperbolic diffeomorphisms can also be found in~\cite{CY} and~\cite{DFPV}.

\subsection{Ergodic theory for flows\label{ss.ergodictheory}}

In this subsection, we state the results of ergodic theory for flows which will be used later. Throughout this
subsection, $\Lambda$ denotes a compact invariant set of the flow $\phi_t$, and $\mu$ is a {\it non-trivial} 
invariant measure of $\phi_t$, i.e., $\mu(\Sing(X))=0$.

\begin{definition}
We say $\Lambda$ admits a {\it dominated splitting} $E\oplus F$
if this splitting is invariant for $\Phi_t$, and there exist $C>0$ and $\lambda<1$ such that for every $x \in \Lambda$, 
and every pair of unit 
vectors $u\in E(x)$ and $v\in F(x)$, one has
$$\|(\Phi_t)_x(u)\|\leq C \lambda^t \|(\Phi_t)_x(v)\|\;\; \text{for}\;\; t>0.$$
\end{definition}

We will see that in the above definition, the assumption of the invariance of the 
splitting is not necessary.

\begin{lemma}\label{l.bothdominatedsplitting}

$E\oplus F|_{\Lambda}$ is a dominated splitting for the flow $\phi_t|_\Lambda$ if and only if it is a dominated splitting for the time-one map $f|_{\Lambda}$. Moreover, if $\phi_t|_{\Lambda}$
is transitive, then we have either $X|_{\Lambda\setminus \Sing(X)}\subset E$ or $X|_{\Lambda\setminus \Sing(X)}\subset F$.

\end{lemma}

\begin{proof}
The proof of the `only if' part is trivial. Now we assume $E\oplus F|_{\Lambda}$ is a dominated splitting for
$f|_{\Lambda}$. Then, by the commutative property between $f$ and $\phi_t$, it is easy to see that, the splitting $\Phi_t(E)\oplus 
\Phi_t(F)|_{\Lambda}$ is also a dominated splitting for $f|_{\Lambda}$.

Because the dominated splitting of fixed dimension is unique, we conclude that this splitting $E\oplus F|_{\Lambda}$ is invariant for
$\Phi_t$. Therefore, $E\oplus F|_{\Lambda}$ is also a dominated splitting for $\phi_t|_{\Lambda}$.

Now suppose $\phi_t|\Lambda$ is transitive. Take $x\in \Lambda\setminus \Sing(X)$ such that $\Orb^+(x)$ is dense in $\Lambda$.
If $X(x)\notin E(x)\cup F(x)$, by the domination, for $t$ sufficient large, $X(\phi_t(x))$ is close to $F(\phi_t(x))$.
We take $t_0$ large such that $\phi_{t_0}(x)$ is close to $x$, then $X(\phi_{t_0}(x))$ is close to $X(x)$ and $F(\phi_{t_0}(x))$
is close to $F(x)$, which implies in particular that $X(x)$ is arbitrarily close to $F(x)$, a contradiction.

Because $\Orb^+(x)$ is dense, by the continuation of flow direction and the subbundles $E$ and
$F$, for $X(x)\in E(x)$ or $X(x)\in F(x)$, we may show that $X|_{\Lambda\setminus \Sing(X_t)}\subset E$ or 
$X|_{\Lambda\setminus \Sing(X_t)}\subset F$ respectively. The proof is complete.
\end{proof}

\begin{definition}
The {\it topological entropy} ({\it resp. metric entropy}) {\it of a continuous flow} equals to the topological entropy
(resp. metric entropy) of its time-one map.
A flow is $\vep$-{\it entropy expansive} if its time-one map is $\vep$-entropy expansive.
\end{definition}

\begin{lemma}\label{l.entropyforflow}

Let $\mu$ be an ergodic invariant measure of $\phi_t$, and $\tilde{\mu}$ be an ergodic component of $\mu$ for
the diffeomorphism $f$. Then $h_{\mu}(\phi_t)=h_{\tilde{\mu}}(f)$.

\end{lemma}

\begin{proof}

Observe that $\tilde{\mu}_t=(\phi_t)_* \tilde{\mu}$ is also an $f$-invariant measure
and 
$$\mu=\int_{[0,1]} (\phi_t)_* \tilde{\mu}_t dt.$$
It is easy to show $h_{\tilde{\mu}_t}(f)=h_{\tilde{\mu}}(f)$ by the following observation:
for any partition $\cA=\{A_1,\dots, A_k\}$, write $\cA_t=\{\phi_t(A_1),\dots, \phi_t(A_k)\}$, 
then $\tilde{\mu}(A_i)=\tilde{\mu}_t(\phi_t(A_i))$. 

Because metric entropy is an affine function respect to the invariant measures, 
$$h_{\mu}(f)=\int_{[0,1]}h_{\tilde{\mu}}(f)dt=h_{\tilde{\mu}}(f)=h_{\tilde{\mu}}(\phi_t).$$
\end{proof}

As a corollary of the above lemma, we state the following two results in the version for flows.

\begin{lemma}\label{l.flowexpansive}({\cite{Bow72b}})
Suppose the flow $\phi_t$ is entropy expansive, then the metric entropy function is upper semi-continuous.
In particular, there exists a maximal measure.

\end{lemma}

\begin{lemma}\label{l.flowrobustranstive}(\cite{LVY}[Lemma2.3])
Let $\cU$ be a $\C^1$ open set of flows which are $\delta$-entropy expansive for some $\delta>0$. Then
the topological entropy varies in an upper semi-continuous manner for flows in $\cU$.

\end{lemma}

The {\it linear Poincar\'e flow $\psi_t$} is defined as following. Denote the normal bundle
of $\phi_t$ over $\Lambda$ by
$$\cN_{\Lambda}=\cup_{x\in \Lambda\setminus \Sing(X)}\cN_x,$$
where $\cN_x$ is the orthogonal complement of the flow direction $X(x)$, i.e.,
$$\cN_x=\{v\in T_xM^d: v\bot X(x)\}.$$
Denote the orthogonal projection of $T_xM$ to $\cN_x$ by $\pi_x$.
Given $v\in\cN_x, x\in M^d\setminus \Sing(X)$, $\psi_t(v)$ is the orthogonal projection of $\Phi_t(v)$ 
on $\cN_{\phi_t(x)}$ along the flow direction, i.e.,
$$\psi_t(v)=\pi_{\phi_t(x)}(\Phi_t(v))=\Phi_t(v)-\frac{<\Phi_t(v),X(\phi_t(x))>}{\|X(\phi_t(x))\|^2}X(\phi_t(x)),$$
where $<.,.>$ is the inner product on $T_xM$ given by the Riemannian metric. The following is
the flow version of Oseledets theorem.

\begin{proposition}\label{p.Oseledets}

For $\mu$ almost every $x$, there exist $k=k(x)\in \mathbb{N}$,
real number $\hat{\lambda}_1(x)>\dots>\hat{\lambda}_k(x)$ and a measurable splitting
$$\cN_x=\hat{E^1_x}\oplus \dots \oplus \hat{E^k_x}$$
such that this splitting is invariant under $\psi_t$, and
$$\underset{t\to \pm\infty}{\lim} \frac{1}{t}\log \|\psi_t(v_i)\|=\hat{\lambda_i}(x)\;\; \text{for every non-zero}\;\; v_i\in \hat{E}^i(x).$$

\end{proposition}

Now we state the relation between Lyapunov exponents and the Oseledets splitting for $\phi_t$ and for $f$:

\begin{theorem}\label{t.exponentsdiffandflow}
For $\mu$ almost every $x$, denote
$\lambda_1(x)>\dots>\lambda_k(x)$ the Lyapunov exponents and the Oseledets splitting
$T_xM=E^1_x\oplus \dots \oplus E^k_x$ of $\mu$ for $f$. Then
$\cN_x=\pi_x(E^1_x)\oplus \dots \oplus \pi_x(E^k_x)$ is the Oseledets splitting of $\mu$ for $\phi_t$. And
the Lyapunov exponents of $\mu$ counted with multiplicity for the flow $\phi_t$ is a subset of the 
exponents of $f$ by removing one of the zero exponent which comes from the flow direction.

\end{theorem}

\begin{definition}

$\mu$ is called a {\it hyperbolic measure for flow $\phi_t$} if it is an ergodic measure of $\phi_t$ 
and all the exponents are non-vanishing. We call the number of the negative exponents of $\mu$, 
counted with multiplicity, its {\it index}.

\end{definition}

\subsection{Lorenz-like classes}

We will prove that Lorenz-like class is always entropy expansive. This is a generalization of the result 
in~\cite{APPV09}, where expansiveness has been shown for every 3-dimensional Lorenz attractor.

\begin{proposition}\label{p.lorenzexpansive}

Let $\Lambda$ be a Lorenz-like class. Then there are $\delta>0$ and a
neighborhood $U$ of $\Lambda$, such that $\phi_t|_{U}$ is $\delta$-entropy expansive. Moreover, 
if $\Lambda$ is a quasi attractor and $Sing(X)\cap\Lambda$ consists of only hyperbolic singularities, then there is a $\C^1$ 
neighborhood $\cU$ of $X$, such that the flows generated by vector fields in $\cU$
restricted on $U$ are all $\delta$-entropy expansive.

\end{proposition}

The proof of Proposition~\ref{p.lorenzexpansive} will be given in Section~\ref{s.entropy}.
By Lemmas~\ref{l.flowexpansive} and~\ref{l.flowrobustranstive}, entropy expansiveness implies
the metric entropy and topological entropy vary in an upper semi-continuous manner. This fact is important
for us to prove that the topological entropy of any Lorenz-like class is positive. The lower
semi-continuation of the topological entropy comes from a similar argument as Katok~\cite{K80} 
and a shadowing lemma of Liao~\cite{L89}, which permits the pseudo orbit to pass near
singularities. 

\begin{theorem}\label{t.lowercontinuiation}

Let $\Lambda$ be a compact invariant set of flow $\phi_t$ and admit a dominated splitting
$E\oplus F$. Suppose $\mu$ is a non-trivial hyperbolic ergodic measure of $\phi_t$ with 
positive entropy and index $dim(E)$. Then for any $n>0$, there 
is a hyperbolic set $\Lambda_n$ arbitrarily close to $\supp(\mu)$ such that 
$h_{top}(\phi_t|_{\Lambda_n})>h_{top}(\phi_t|_{\supp(\mu)})-\frac{1}{n}.$

\end{theorem}

Theorem~\ref{t.lowercontinuiation} and the Corollary~\ref{c.homoclinicclass} are proved in
Section~\ref{s.lowercontinuiation}.

\begin{remark}\label{r.katokshadowing}

We need to note that, the original argument of~\cite{K80} cannot be applied directly on flows, 
even for the flows which are uniformly hyperbolic. The obstruction comes from the difference of the shadowing lemmas for 
diffeomorphisms and for flows. In the latter case, the period of periodic pseudo orbit and the period
of the shadowing periodic orbit are in general different. The key point of our argument is that, 
we need an uniform estimation on the difference of the times-- (d) of Theorem~\ref{t.shadowing}.
As far as the author knows, this estimation is new.

\end{remark}

Immediately, we have the following corollary, which is similar to~\cite{SGW}[Theorem 5.6].
The difference is that here we permit the singularities to be non-hyperbolic.

\begin{corollary}\label{c.homoclinicclass}

Under the assumptions of Theorem~\ref{t.lowercontinuiation}, the chain recurrent class of $\Lambda$ 
contains a non-trivial homoclinic class.

\end{corollary}

We need the following theorem from~\cite{SGW}:

\begin{theorem}[~\cite{SGW} Theorem C]\label{t.Lorenzlike}

There is a residual subset $\cR_1$ of star flows, such that for every flow contained in $\cR_1$, its non-trivial
Lyapunov stable chain recurrent classes are all Lorenz-like.

\end{theorem}

In Appendix~\ref{apendix.isolation}, we will also show that:

\begin{theorem}\label{t.isolated}

There is residual subset $\cR_2$ of $\C^1$ flows, such that for every flow contained in $\cR_2$, if 
$\Lambda$ is a Lorenz-like class which contains a periodic orbit, then it is an attractor.

\end{theorem}

\section{Positive entropy\label{s.positiveentropy}}

In this section, we give the proof of Theorem~\ref{t.Flowboundary}, Corollaries~\ref{c.Fentropyconjecture} 
and~\ref{c.transtive}. We assume $g$ is a diffeomorphism which admits a dominated splitting $E\oplus F$, 
$a$ and $\vep_0$ are defined as in Subsection~\ref{ss.dominated splitting}. $K_0$ is given by Lemma~\ref{l.onesteplocalvolume}.

\begin{lemma}\label{l.localvolume}

Suppose $D$ is a disk with dimension $dim(F)$ and tangent to local $F$ cone. Then 
for any $x\in D$ and $\vep<\frac{\vep_0}{\|g\|_{\C^1}}$, one has $\Leb(g^n(D\cap B_n(x,\vep)))
\leq K_0$.

\end{lemma}

Without the assumption on dominated splitting, the previous lemma is false even for analytic 
diffeomorphism, e.g., see~\cite{BLY}.

\begin{proof}

By the definition of dynamical ball $B_n(x,\vep)$ and Lemma~\ref{l.localtangency}, 
$D_n=f^n(D\cap B_n(x,\vep))$ is tangent to local $F$ cone and is contained in $B_\vep(x)$. 
By Lemma~\ref{l.onesteplocalvolume}, the volume of $D_n$ is bounded by $K_0$.

\end{proof}

\begin{remark}\label{r.SX}

As we mentioned in the introduction, the proofs of~\cite{SX10} about Theorem~\ref{t.Flowboundary}
and Corollary~\ref{c.Fentropyconjecture} do not work. A main reason is because in their proof of~\cite{SX10}[Proposition 2], 
they need $f^n(D\cap B_n(x,\vep))$ to be almost a ball, which occurs only when the bundle $F$ is uniformly 
expanding. In general, this set is even not connected.

\end{remark}

\begin{proof}[Proof of Theorem~\ref{t.Flowboundary} :]

For any $\delta>0$, choose $D$ a disk which is tangent to the $F$ cone and satisfies 
$v_F(D)\geq v_F(f)-\delta$, that is:

$$\limsup_n \frac{1}{n}\log(\Leb(f^n(D)))\geq v_F-\delta.$$

Take $\vep>0$ and $\Gamma_n=\{x_1,\dots,x_{r_n(D,\vep)}\}$ an ($n,\vep$)-spanning set of 
$D$. By Lemma~\ref{l.localvolume}, $\Leb(f^n(D\cap B_n(x,\vep)))\leq K_0$.
Observe that $\{B_n(x_i,\vep)\}_{i=1}^{r_n(D,\vep)}$ is a cover of $D$, we have
$$\Leb(f^n(D))\leq \sum_i \Leb(f^n(D\cap B_n(x_i,\vep)))\leq K_0r_n(D,\vep).$$
Hence,
\begin{equation*}\label{eq:12}
\begin{split}
& \limsup_n \frac{1}{n}\log (r_n(D,\vep))\geq \limsup \frac{1}{n}\log(\Leb(f^n(D))/K_0) \\
& = \limsup \frac{1}{n}\log(\Leb f^n(D)) \geq v_F-\delta. \\
\end{split}
\end{equation*}
This implies that $h_{top}(g,D)\geq v_F-\delta$.
Since $\delta$ can be chosen arbitrarily small, we conclude the proof of Theorem~\ref{t.Flowboundary}.

\end{proof}

Now let us begin the proof of Corollaries~\ref{c.Fentropyconjecture} and~\ref{c.transtive}:

\begin{proof}[Proof of Corollary~\ref{c.Fentropyconjecture}:]

It is well known that the supremum of volume expansion of disks with dimension $dim(F)$ 
is larger than or equal to $\sp(f_{*,dim(F)})$ (for example, see p.287~\cite{Yom87}). 
Now we will show that the same proof also provides that
$$v_F(f)\geq \sp(f_{*,dim(F)}).$$

As explained in~\cite{Yom87}: {\it the norm of the homology class is bounded by integrals 
of some fixed differential forms on this class; these integrals, in turn, are bounded by 
the volume of the chain, representing this class.} We may choose a chain such that every 
simplex is transverse to the $E$ bundle. Because of the dominated splitting, under the 
iteration of $f$, the image of every simplex is finally tangent to the $F$ cone. Hence, 
their volume expansion speed will be bounded by $v_F(f)$. Then the previous argument 
still works here.

As a corollary of Theorem~\ref{t.Flowboundary}, we show that $h_{top}(f)\geq \sp(f_{*,dim(F)})$. 

\end{proof}

\begin{proof}[Proof of Corollary~\ref{c.transtive}:]

Because $f$ is robustly transitive, by~\cite{BDP03}, $f$ admits a dominated splitting 
$E_1\oplus \dots \oplus E_k$ where $E_1$ is volume contracting and $E_k$ is volume 
expanding. More precisely, there is $\lambda<1$ and $C>0$ such that for any $x\in M$

$$\det(Df^n|_{E_1(x)})< C \lambda^n\;\;\; \text{and}\;\;\; \det(Df^n|_{E_k(x)})> C \lambda^{-n}.$$ 

Now we claim that $v_{E_k}(f)\geq -\log\lambda$. Then by Theorem~
\ref{t.Flowboundary}, we finish the proof.

It remains to prove this claim. Take $D$ any disk tangent to $F$ cone. By the invariance 
of $F$ cone, $f^n(D)$ is still tangent to $F$ cone for any $n>0$. Then the determine 
of the tangent map of $f^n|D$ is close to $\det(Df^n|_F)$, which means that 

$$\Leb(f^n(D))\geq C\lambda^{-n}.$$

Hence $v_F(D)=\limsup \frac{1}{n}\log \Leb(f^n(F)) \geq -\log \lambda$.

\end{proof}

\section{Entropy expansiveness \label{s.entropy}}

Throughout this section, let $\Lambda$ be a Lorenz-like class which admits 
a sectional-hyperbolic splitting $E^s\oplus F^{cu}$.

The structure of this section is the following. In Subsection~\ref{ss.mainA}, we use 
Proposition~\ref{p.lorenzexpansive} to prove Theorem~\ref{main.A}. The proof of 
Proposition~\ref{p.lorenzexpansive} on the entropy expansiveness of Lorenz-like class 
will be given in Subsection~\ref{ss.entropyexpansiveness}.

\subsection{Proof of Theorem~\ref{main.A}\label{ss.mainA}}

First let us recall the proof of Theorem~\ref{t.Flowboundary}, where we only use the forward invariance of 
the $F$ cone. In a small attracting neighborhood of $\Lambda$, the 
$F$ cone is well defined and forward invariant, hence, Theorem~\ref{t.Flowboundary} still works.

\begin{proof}[Proof of Theorem~\ref{main.A}:]
By the definition of Lorenz-like class, $\Lambda$ is a compact invariant set of the time-one map
$f$ and admits a dominated splitting $E^s\oplus F^{cu}$, where $F^{cu}$ is volume expanding. 
We take a forward invariant $F^{cu}$ cone on $\Lambda$.

Let $\{U_i\}_{i=1}^\infty$ be a sequence of neighborhoods of $\Lambda$ as in
Definition~\ref{d.lorenzlike}, i.e., they satisfy:
\begin{itemize}

\item $U_1\supset U_2\dots$ and $\cap U_i=\Lambda$;

\item $f^n(U_{i+1})\subset int(U_i)$ for any $i,n\in \mathbb{N}$.

\end{itemize}

Assume 
$U_1$ is sufficiently small, we may extend the $F^{cu}$ cone to $U_1$ which is still forward 
invariant. Then for any disk $D\subset U_1$ which is tangent to $F$ cone, $f^n(D)$ is 
still tangent to $F$ cone. Moreover, because $F^{cu}|_{\Lambda}$ is volume expanding, there 
exist $C>0$ and $0<\lambda<1$ such that:

$$\Leb(f^n(D))\geq C\lambda^{-n}.$$ 
This implies $v_F(D)\geq -\log \lambda$.

For each $i\in \mathbb{N}$, applying Theorem~\ref{t.Flowboundary} on $f|_{U_i}$, we deduce that $h_{top}(f|_{U_i})\geq
-\log \lambda$. By the variation principle, there is an $f$ ergodic invariant measure 
$\mu_i$ supported on $U_i$ with $h_{\mu_i}(f)\geq -\log \lambda-\frac{1}{i}$. Replacing by a subsequence, 
we suppose that $\mu_i\to\mu_0$ in the weak-* topology, where $\mu_0$ is an invariant 
measure supported on $\Lambda$. As shown in Lemma~\ref{l.flowexpansive}, entropy expansiveness implies 
the metric entropy function is upper semi-continuous, this implies that
$h_{\mu_0}(f)\geq-\log \lambda$. Therefore, by the definition of topological entropy for flows, we conclude that
$$h_{top}(\phi_t|_{\Lambda})= h_{top}(f|_{\Lambda})\geq h_{\mu_0}(f)\geq-\log \lambda,$$
the proof is complete.
\end{proof}

\subsection{Entropy expansiveness\label{ss.entropyexpansiveness}}

Let $U$ be a small neighborhood of $\Lambda$, such that the maximal invariant 
set of $\phi_t|_{U}$, denoted by $\tilde{\Lambda}$, is sectional-hyperbolic. This means, 
$\phi_t|_{\tilde{\Lambda}}$ admits a dominated splitting $E^{s}\oplus F^{cu}$ and there is 
$0<\lambda_0<1$ such that, for any 2-dimensional subspace $\Sigma_x\subset F^{cu}(x)$ ($x\in \tilde{\Lambda}$), 
we know that $\det(Tf_x|_{\Sigma_x})>\frac{1}{\lambda_0}$. We may assume the above inequality 
also holds for any 2-dimensional subspace contained in $2a_0$ $F^{cu}$-cone for 
$a_0$ sufficiently small. In order to simplify the notation, we write briefly $x_t=\phi_t(x)$.

The proof of Proposition~\ref{p.lorenzexpansive} occupies the rest of this subsection. 
We need notice that although this result holds for flow, in the proof, we will only consider 
the time-one diffeomorphism. The main reason is because we need use the fake foliations, 
which are locally invariant for the diffeomorphism, but not for the flow! Now let us sketch the proof.

Although the singularities can be non-isolated--by non-hyperbolicity,  
in Subsubsection~\ref{sss.singularity} we will show that typical points of $\tilde{\Lambda}$ 
can only approach finitely many singularities in $\Sing(X)$. Let us explain this fact with more details. 
It will be shown that $f|_{\Sing(X)\cap\tilde{\Lambda}}$ admits a partially hyperbolic splitting with a 
1-dimensional center direction, and the singularities are either isolated or contained 
in finitely many center segments. Moreover, typical points of $\tilde{\Lambda}$ cannot 
approach the interior of each center segment, and when a singularity is approached by a typical point 
from one direction along the center, then the corresponding branch of center leaf is (topologically) 
contracting and with uniform size.

As a standard argument in~\cite{LVY}, the infinite Bowen ball is contained in a fake $cu$ 
leaf. In Subsubsection~\ref{sss.expandinginculeaf}, we will see that the distance between 
different flow orbits contained in the same infinity Bowen ball is expanding. There are two
kinds of analysis: the orbits are away from singularities or close to singularities. We need notice the
reader that, in the latter case, we cannot use linearization for the dynamics close to a singularity,
because which can not be hyperbolic. In fact, a new argument is used to analyze these orbits
which are close to singularities. More precisely, we use two different kinds of fake foliations-- 
the fake foliations defined in a neighborhood of singularities which come from the partially 
hyperbolic splitting on $\Sing(X)$; and the fake foliations come from the dominated splitting 
on $\tilde{\Lambda}$--to build a local product structure in $cu$ leaves.

In Subsubsection~\ref{sss.entropyexpansive}, we give the proof of Proposition~\ref{p.lorenzexpansive}.

\subsubsection{Singularities and center segments\label{sss.singularity}}

\begin{lemma}\label{l.singularities}

$f|_{\Sing(X)\cap \tilde{\Lambda}}$ admits a partially hyperbolic splitting $E^s\oplus E^c\oplus E^u$, where $E^c$ 
is a 1-dimensional bundle corresponding to an eigenvalue with norm less than or equal to one.

\end{lemma}

\begin{proof}

We will only prove that the singularities contained in $\Lambda$ admit such a partially hyperbolic splitting. 
This is because when we choose $U$ sufficiently small, $\Sing(X)\cap\tilde{\Lambda}$ is contained in a small neighborhood 
of $\Sing(X)\cap \Lambda$, and then induced the same partially hyperbolic splitting.
 
Recall that $\Lambda$ admits a dominated splitting $E^s\oplus F^{cu}$ where $F^{cu}$ is sectional-expanding. It suffices to prove that for any singularity $\sigma\in \Lambda$, $Tf|_{F^{cu}(\sigma)}$ has 
one eigenvalue with norm less than or equal to one. Suppose this fact and denote $E^c$ the direction of 
the eigenvector. Then the splitting $F^{cu}|_{\Sing(X)\cap \Lambda}=E^c\oplus E^u$ is continuous, invariant, 
and the latter subbundle is uniformly expanding. Therefore this splitting is a dominated splitting, 
and $E^s\oplus E^c\oplus E^u$ is the corresponding partially hyperbolic splitting on $\Sing(X)\cap \Lambda$.

It remains to show that for any singularity $\sigma\in \Lambda$, $Tf|_{F^{cu}(\sigma)}$ has a eigenvalue with norm 
less than or equal to one. Suppose by contradiction that there is a singularity $\sigma\in \Lambda$ such that
$Tf|_{F^{cu}(\sigma)}$ is uniformly expanding. Then $\sigma$ is an isolated hyperbolic saddle of $f$. We claim that $\Lambda\cap (W^s(\sigma)\setminus \sigma)\neq \emptyset$.

The proof of this claim is standard. By the definition of Lorenz-like class, $\Lambda$ contains a regular
point $x$. Then for any $n\in \mathbb{N}$, there is a $\frac{1}{n}$ pseudo orbit $\{\phi_t(x); 0\leq t\leq t_1\};\dots;\{\phi_t(x^{k_n});0\leq t\leq t_{k_n}\}$ contained in $\Lambda$ with $t_1,\dots,t_{k_n}\geq 1$ 
such that $d(\phi_{t_{k_n}}(x^{k_n}),\sigma)\leq \frac{1}{n}$. Let $y^n$ be the last time this pseudo orbit 
enters $B_\vep(\sigma)$ for some small $\vep$. By taking a subsequence, we may assume that $\lim y^n=y^0$. 
Then $y^0\in \Lambda\setminus \sigma$ and $\phi_t(y^0)\in B_\vep(\sigma)$ for any $t>0$. This implies 
that $y^0\in W^s(\sigma)\setminus \sigma$, which proves the claim.

By the invariance of the stable manifold of $\sigma$, $\phi_t(y^0)\subset W^s(\sigma)$ for $t\in \mathbb{R}$. 
This implies in particular that $X(y^0)$ is contained in the $E^s$ cone. Take a local orbit $l= (\phi_t(y^0))_{t\in [-\delta,\delta]}$, which is tangent to $E^s$ cone. By the expansion of vectors in $\cC(E^{s})$ under the iteration of $f^{-1}$, $\length(f^{-n}(l))\rightarrow \infty$, which contradicts the fact that
\begin{equation}\label{e.bounded}
\length(f^{-n}(l))\leq \frac{\max\{\|X(x)\|; x\in \Lambda\}}{\min\{\|X(y)\|; y\in l\}}\length(l)
\end{equation}
is uniformly bounded.
\end{proof}




The bundles in a dominated splitting are in general not integrable. The following lemma was
borrowed from~\cite{LVY}[Lemma 3.3] (see also~\cite{BW}[Proposition 3.1]) which shows that one can always 
construct local fake foliations. Moreover, these fake foliations have local product structure, and this 
structure is preserved as soon as they stay in a neighborhood.

\begin{lemma}\label{l.fakefoliation}[\cite{LVY}[Lemma 3.3]]
Let $K$ be a compact invariant set of $f$. Suppose $K$ admits a dominated splitting $T_KM=E^1
\oplus E^2\oplus E^3$. Then there are $\rho>r_0>0$ such that the neighborhood $B_\rho(x)$ of 
every $x\in K$ admits foliations $\cF^1_x$,$\cF^2_x$,$\cF^3_x$,$\cF^{12}_x$ and $\cF^{23}_x$ 
such that for every $y\in B_{r_0}(x)$ and $*\in\{1,2,3,12,23\}$:

\begin{itemize}

\item[(i)] $\cF_x^*(y)$ is $\C^1$ and tangent to the respective cone.

\item[(ii)] $f(\cF^*_x(y,r_0))\subset \cF^*_{f(x)}(f(y))$ and $f^{-1}(\cF^*_x(y,r_0))\subset \cF^*_{f^{-1}(x)}(f^{-1}(y))$.

\item[(iii)] $\cF^1_x$ and $\cF^2_x$ subfoliate $\cF^{12}_x$ and $\cF^2_x$ and $\cF^3_x$ subfoliate $\cF^{23}_x$. 

\end{itemize}

\end{lemma}

\begin{remark}\label{r.flowfakefoliation}
In the following proof, we will consider two kinds of fake foliations:
\begin{itemize} 
\item The partially hyperbolic splitting $E^s\oplus E^c\oplus E^u$ on the singularities $\Sing(X)\cap \tilde{\Lambda}$
induces fake foliations $\hat{\cF}^i$ $i=s,c,u,sc,cu$, where $\hat{\cF}^s$ and $\hat{\cF}^u$ 
have contracting and expanding property. 

\item The dominated splitting $E^s\oplus F^{cu}$ on $\tilde{\Lambda}$ induces fake foliations 
$\tilde{\cF}^s$ and $\tilde{\cF}^{cu}$. 

\end{itemize}
We note that the proposition (ii) above may not hold for flows, i.e., the fake 
foliation does not preserved by $\phi_t$ when $t\notin \mathbb{Z}$. Moreover, the 
flow orbit is not even locally saturated by $\tilde{\cF}^{cu}$ leaf. In fact, the fake foliations depends on
extension of the dynamics in the tangent bundle (see~\cite{BW}[Proposition 3.1]), which is in general
not preserved by the flow.

\end{remark}

The following lemma provides an important observation on infinity Bowen ball, a proof can be found 
in the proof of \cite{LVY}[Theorem 3.1], which is based on the local product structure and 
hyperbolicity of $\tilde{\cF}^s$:

\begin{lemma}\label{l.bowenball}

For any $x\in\tilde{\Lambda}$, $B_\infty(x,r_0)\subset \tilde{\cF}_x^{cu}(x)$.

\end{lemma}

As an immediately corollary of Lemma~\ref{l.bowenball}, we obtain a kind of weak saturated 
property for $cu$ fake leaf $\tilde{\cF}^{cu}(x)$:

\begin{corollary}\label{c.localcoherence}

There is $r_1>0$ such that for any $y\in B_\infty(x,r_0/2)$, $y_t\in \tilde{\cF}_x^{cu}(x)$ for 
$|t| \leq r_1$.

\end{corollary}

\begin{proof}
Denote $D_0=\max_{x\in M}{\|X(x)\|}$ and $r_1=\frac{r_0}{2D_0}$. Then for any $|t| \leq r_1$, 
the segment of flow orbit between $f^n(y_t)$ and $f^n(y)$ has length bounded by $r_1 D_0\leq r_0/2$. 
Hence, 
$$d(f^n(x),f^n(y_t))\leq d(f^n(x),f^n(y))+d(f^n(y),f^n(y_t))\leq r_0.$$
Which implies that $y_t\in B_\infty(x,r_0)\subset \hat{\cF}^{cu}(x)$ for $|t|\leq r_1$. 

\end{proof}

The discussion of the infinity Bowen ball of a singularity is the same as that in \cite{LVY}[Theorem 3.1]:

\begin{lemma}\label{l.singularitiesexpansive}

There is $L>0$ such that for any singularity $\sigma\in \tilde{\Lambda}$, $B_\infty(\sigma,r_0/2)$ is a single point 
or a 1-dimensional center segment with length bounded by $L$. In the latter case, the center 
segment consists singularities and saddle connections.

\end{lemma}

\begin{remark}\label{r.singularitiesstructure}

By Lemma~\ref{l.singularitiesexpansive}, $\Sing(X)\cap \tilde{\Lambda}\subset \{\tilde{\sigma}_1,\dots,\tilde{\sigma}_m\}
\cup I_1\cup \dots\cup I_t$ where $\{\tilde{\sigma}_1,\dots, \tilde{\sigma}_m\}$ are finite isolated singularities, 
and $\{I_1,\dots, I_t\}$ are center segments tangent to the center direction. Moreover, each center 
segment $I_i$ is fixed by $f$ and contained in $B_\infty(\sigma_i,r_0/2)$ for some singularity $\sigma_i\in I_i$. We assume that different center segments $I_i\neq I_j$ only intersect on the boundary points, which are singularities.
By a double covering, we may always suppose $E^c$ is orientable and call the boundary points of each center segment $I_i$ 
by the {\it left} extreme point $\sigma_i^-$ and the {\it right} extreme point $\sigma_i^+$.

\end{remark}

Now we give more description for the dynamics close to the center segments. Consider the fake leaf 
$\hat{\cF}_{\sigma_i}^{cu}(\sigma_i)$ which contains $I_i$. Define:
$$W^{uu}_{\delta}(I_i)=\cup_{x\in I_i}\hat{\cF}^u_{\sigma_i}(x,\delta)\subset\hat{\cF}_{\sigma_i}^{cu}(\sigma_i)\;\;
\text{and}\;\; \cN_i(\delta)=\cup_{y\in W^{uu}_\delta(I_i)} \hat{\cF}^s_{\sigma_i}(y,\delta).$$

\begin{lemma}\label{l.awayfromIi}

There are $\delta, r_2>0$ and center segments $I_{i,j}\subset I_i$ for $1\leq j \leq k_i$ such that 
$\length(I_{i,j})<r_2$, $\Sing(X)\cap I_i=\Sing(X)\cap(\cup_j I_{i,j})$ and for any non-trivial ergodic measure 
$\mu$ of $\phi_t$, one has $\mu(\cN_{i,j}(\delta))=0$ for each $I_{i,j}$.

\end{lemma}

\begin{proof}

There is $\vep>0$ such that for any $z\in \cF^s_{r_0}(\sigma)\setminus \cF^s_{r_0/2}(\sigma)$
where $\sigma$ belongs to $\Sing(X)\cap\tilde{\Lambda}$, and for any $z^{'}\in B_\vep (z)$, $X(z^{'})$ is tangent to the 
$E^s$ cone.

Take $r_2>0$ such that for any $x,y\in I_i$ with $d(x,y)<r_2$ and for any $z\in \hat{\cF^s}_\sigma
(x,r_0)\setminus \hat{\cF}^s_\sigma(x,r_0/2)$, there is $z^{'}\in \hat{\cF^s}_\sigma(y,r_0)\setminus 
\hat{\cF}^s_\sigma(y,r_0/2)$ satisfying $d(z,z^{'})<\vep$.

Recall that $I_i\subset B_\infty(\sigma_i,r_0/2)$ is invariant and $f|I_i$ is partially hyperbolic. 
By the stable manifold theorem 
(\cite{HPS77}) for normally hyperbolic submanifold $I_i$, for any $x\in I_i$, $\hat{\cF}_\sigma^s(x)$ is 
the local strong stable manifold of $x$. Similar property holds for the local strong unstable manifold. Moreover,
$$W^s_{loc}(I_i)=\cup_{x\in I_i}\hat{\cF}_\sigma^s(x,r_0) \;\;\; \text{and}\;\;\; W^u_{loc}(I_i)=\cup_{x\in I_i}\hat{\cF}_\sigma^u(x,r_0).$$

Choose finite sub-segments $(I_{i,j})_{j=1}^{k_i}$ such that 
\begin{itemize}

\item $\length(I_{i,j})<r_2$;

\item the boundary points of each $I_{i,j}$ are singularities, and different subsegments only intersect on the boundary points;

\item $\Sing(X)\cap I_i=\Sing(X)\cap(\cup_j I_{i,j})$.

\end{itemize}

Suppose there are $\delta_n\rightarrow 0$ such that $\mu(\cN_{i,j}(\delta_n))>0$. 
Since $\mu$ is a non-trivial ergodic measure, $\mu(W^s_{loc}(I_{i,j}))=
\mu(W^u_{loc}(I_{i,j}))=0$. There exists
$$x^n\in \cN_{i,j}(\delta_n)\setminus (W^s_{loc}(I_{i,j})\cup W^u_{loc}(I_{i,j})).$$ 
By~\cite{HPS77}, the negative iteration of $x^n$ must leave $\cN_{i,j}(r_0)$. 
Take $t_n<0$ the last time such that $f^{m}(x^n)\in \cN_{i,j}(r_0)$ for any $t_n\leq m \leq 0$. Then 
$t_n\rightarrow -\infty$ and we may suppose 
$$f^{-t_n}(x^n)\rightarrow x^*\in W^s_{loc}(I_{i,j})\setminus I_{i,j}.$$
As discussed before, for $n$ is sufficiently large, 
$X(f^{t_n}(x^n))$ is tangent to $E^s$ cone. Denote by $l_\delta=(\phi_{t}(f^{t_n}(x_n)))_{|t|\leq \delta}$. Then
for $\delta$ sufficiently small, $l_\delta$ is tangent to the $E^s$ cone. It follows easily that
$\lim_{m\to \infty}\length f^{t_n-m}(l_\delta)\to \infty$, which contradicts the fact that the 
length of any image of a flow orbit with bounded length is always bounded, as explained in the last 
paragraph of the proof of Lemma~\ref{l.singularities}.
\end{proof}

From now on, we treat every isolated singularity as a trivial center segment. Re-order the center segments
$\{I_{i,j}\}$ by $\{I_i\}_{1\leq i \leq t}$ and choose a singularity $\sigma_i$ inside each $I_i$. From the construction,
$\length(I_i)<r_2$ for every $I_i$. Lemma~\ref{l.awayfromIi} shows that, for any non-trivial 
ergodic measure supported in $U$, its generic point can only approach the extreme singularities of every center segment
from their {\it half neighborhoods} which is defined as following. Let $\sigma_i^+$ be a
right extreme point of $I_i$, $\cup_{y\in \hat{\cF}^u_{\sigma_i}(\sigma_i^+,r_0)}\hat{\cF}^s_{\sigma_i}(y,r_0)$ 
is a topological codimension-one disk, which separates $B_{r_0}(\sigma^+_i)$ into two connect components, the 
{\it right neighborhood of $\sigma_i^+$}, denoted by $B^r_{r_0}(\sigma_i^+)$, is just the the right component. 
In a similar way, we may define the {\it left neighborhood} $B^l_{r_0}(\sigma_i^-)$ for 
the left extreme point $\sigma_i^-$.

In the following proof, we only consider the right extreme points, similar results always hold for the left 
ones.

\begin{lemma}\label{l.extremepoints}

There is $\delta>0$ such that for every non-trivial 
ergodic measure $\mu$ of $\phi_t$ and for every center segment $I_i$, 
if $\mu(B^r_{\delta}(\sigma^+_i))>0$, then the right branch of the center leaf 
$\hat{\cF}^c_{\sigma_i}(\sigma^+_i,r_2)$, denoted by $\hat{\cF}^{c,+}_{\sigma_i}(\sigma^+_i,r_2)$, is topologically contracting.

\end{lemma}

\begin{proof}

Since there are only finitely many extreme points, it suffices to prove only for one right extreme point $\sigma_i^+$.

First observe that there are no singularities contained in $\hat{\cF}^{c,+}_{\sigma_i}
(\sigma^+_i,r_2)$. Suppose such a singularity does exist, there is a center segment connecting this singularity and
$\sigma_i^+$ from the right. Then the proof of Lemma~\ref{l.awayfromIi} implies that generic point of $\mu$ cannot 
approach $\sigma_i^+$ from the right, a contradiction.

Because $\hat{\cF}^{c,+}_{\sigma_i}(\sigma^+_i,r_2)$ is invariant and contains no singularities, it is topologically 
expanding or contracting. Assume it is topologically expanding, then $\hat{\cF}^{cu,+}_{\sigma_i}(\sigma_i^+,r_2)$ 
belongs to the unstable set of $\sigma_i^+$. We claim that there is $\delta^{'}>0$ such that 
$\mu(B^r_{\delta^{'}}(\sigma_i^+))=0$. Then we may take $\delta<\delta^{'}$ and conclude the proof.

It remains to prove this claim. Suppose there are positive numbers $\lim \delta_n^{'}\to 0$ such that $\mu(B^r_{\delta^{'}_n}(\sigma_i^+))>0$.
Because $\mu$ is non-trivial, we may assume there are $\mu$ generic points 
$$x^n\in B^r_{\delta^{'}_n}(\sigma_i^+)\setminus (W^u(\sigma^+_i)\cup \cF^s(\sigma_i^+)).$$
Consider $t_n<0$ the last time such that $f^{m}(x^n)\in B^r_{r_0}(\sigma^+_i)$ for any $t_n \leq m \leq 0$.
It is easy to see that $t_n\rightarrow \infty$. By the topological expansion of $\hat{\cF}^{cu,+}_{\sigma_i}(\sigma_i^+,r_2)$, 
we may suppose $f^{-t_n}(x^n)\rightarrow x^*\in \cF^s(\sigma^+_i)$.

The rest proof is the same as the last paragraph of the proof of Lemma~\ref{l.singularities}: 
for $n$ sufficiently large, $f^{t_n}(x^n)$ is close to $x^*$ and $X(f^{-t_n}(x^n))$ is tangent 
to $E^s$ cone. This is a contradiction, since the length of a small segment of the flow orbit which contains 
$f^{-t_n}(x^n)$ grows exponentially fast. The proof is complete.
\end{proof}

Lemmas~\ref{l.awayfromIi} and~\ref{l.extremepoints} explain to us that the generic point of any 
non-trivial ergodic measure can only approach finitely many half neighborhood of singularities which are
topologically `hyperbolic'.

Take $\delta_0<r_2$ satisfying Lemmas~\ref{l.awayfromIi} and~\ref{l.extremepoints}. From now on, 
we rename the half neighborhoods of the extreme points by $B_{\delta_0}^h(\sigma_1),\dots, B_{\delta_0}^h(\sigma_k)$. 
Note that, in this notation, when we are talking about the two half neighborhoods of the same isolated 
singularity, we give different name for this singularity.

\subsubsection{Expanding in a $cu$ fake leaf\label{sss.expandinginculeaf}}

In this subsubsection, we will build the expansion of the distance between different flow orbits 
contained in the same infinity Bowen ball. The discussion can be divided into two cases: 
the orbits are away from singularities (Lemma~\ref{l.localholonomy}), and are close 
to singularities (Lemma~\ref{l.nearsingularity}). We need the following notation:

\begin{definition}

For $x\in U\setminus \Sing(X)$, write $B_\delta^{\bot}(x)=\exp_x(\cN_x(\delta))$. 
And for $y\in B_\delta^{\bot}(x)\cap \tilde{\cF}^{cu}(x)$, denote by
$d_x^*(x,y)$ the distance between $x$ and $y$ in the submanifold
$B_\delta^{\bot}(x)\cap \tilde{\cF}^{cu}(x)$.

\end{definition}

\begin{lemma}\label{l.distancecompactable}

For any $\vep>0$ and $L>0$, there is $\delta>0$ such that for every $x\in \tilde{\Lambda} \setminus B_\vep(\Sing(X))$ 
and $y\in \tilde{\cF}_x^{cu}(x)\cap B^\bot_{\delta}(x)$, we have
$$d(x,y)<d^*_x(x,y)<Ld(x,y).$$

\end{lemma}

\begin{proof}

This comes from the continuity of $\tilde{\cF}^{cu}_{x}(x)$ in the $\C^1$ topology.

\end{proof}

\begin{definition}\label{d.holonomy}

For any regular point $x\in U$, denote $P_x$ the projection along the flow:
$$P_x: B_\delta(x)\rightarrow B_{r_0}^{\bot}(x)$$
which is defined in a neighborhood of $x$. For any point $y$ in this neighborhood, define
$t_x(y)$ the time which satisfies $y_{t_x(y)}=P_x(y)$.

\end{definition}

\begin{lemma}\label{l.localholonomy}

For $\vep>0$ and $1<b_0<\frac{1}{\lambda_0}$, there is $\delta>0$ such that for any $x$ satisfying 
$(x_t)_{t\in[0,1]}\subset \tilde{\Lambda}
\setminus B_{\vep}(\Sing(X))$, and for any $y\in B_\infty(x,\delta)$:  $$d_{x_1}^*(x_1,P_{x_1}(y_1))>b_0\frac{\|X(x)\|}{\|X(x_1)\|}d^*_x(x,P_x(y)).$$

\end{lemma}

\begin{proof}
Write $C=(\frac{1}{\lambda_0b_0})^{1/6}$ and take $0<t_0<\min\{\frac{r_0}{2},\frac{r_1}{2}\}$ such that for any
2-dimensional subspace $\Sigma$ in the tangent space, $1/C<\Jac(\Phi_t|_{\Sigma})<C$ for any $|t|<t_0$. 
For two vectors $u,v\in T_xM$, denote by $\cP[u,v]$ the parallelogram defined by these two vectors 
and $\cA(u,v)$ its area.

When $\delta>0$ is small, for $y\in B_\infty(x,\delta)$, $|t_x(y)|, |t_{x_1}(y_1)|<t_0$. Then by Corollary~\ref{c.localcoherence}, 
$P_{x}(y)\in \tilde{\cF}^{cu}_x(x)\cap B^\bot_{r_0}(x)$
and $P_{x_1}(y_1)\in \tilde{\cF}^{cu}_{x}(x)\cap B^\bot_{r_0}(x)$. 

There is $0<\vep_1\ll r_0/2$ which depends on $\vep$ and $b_0$, such that 
for any $z\in B^{\bot}_{\vep_1}(x_t)$, the following conditions are satisfied:

\begin{itemize}

\item[(a)] $\frac{1}{C}<\frac{\|X(z)\|}{\|X(x_t)\|}<C$.

\item[(b)] For $v\in T_z B^\bot_{\vep_1}(x_t)$, we have $\frac{1}{C}\|v\|\|X(z)\|\leq \cA(v,X(z))\leq C\|v\|\|X(z)\|$.
\end{itemize}

We may further suppose that $\delta$ is sufficiently small, such that
$d^*_x(x,P_x(y))<\vep_1$, $d^*_{x_1}(x_1,P_{x_1}(y_1))<\vep_1$ and by Lemma~\ref{l.distancecompactable}:

\begin{equation}\label{eq:compatable}
\begin{split}
d(x,P_x(y))& <  d^*_x(x,P_x(y))<Cd(x,P_x(y)),\\
d(x_1,P_{x_1}(y_1))& < d^*_{x_1}(x_1,P_{x_1}(y_1))<Cd(x_1,P_{x_1}(y_1)).\\ 
\end{split}
\end{equation}
Moreover, take a curve $\tilde{l}_1\subset B^\bot_{\vep_1}(x_1)\cap \hat{\cF}^{cu}_{x_1}(x_1)$ which links $x_1$ and $P_{x_1}(y_1)$
with $\length(\tilde{l}_1)=d_{x_1}^*(x_1,P_{x_1}(y_1)))$, and write $\tilde{l}_0=f^{-1}(\tilde{l}_1)$ and $l=P_{x}(\tilde{l}_0)$, we suppose $|t_x|_{\tilde{l}_0}|\leq t_0$ and $l\subset B^{\bot}_{\vep_1}(x)$.
Then $l$ is a curve contained in $B_{\vep_1}^\bot(x)$ which connects $x$ and $P_x(y)$.
We note that although by the local invariance, $\tilde{l}_0\subset \tilde{\cF}^{cu}_{x}(x)$, 
$l$ is not necessary contained in $\tilde{F}^{cu}_{x}(x)$.
Denote $H$ the smooth holonomy map between $\tilde{l}_1$ and $l$ which is induced by flow, then 
$H(P_{x_1}(y_1))=P_x(y)$.
We claim that 
$$\|dH\|\leq \lambda_0 C^5\frac{\|X(x_1)\|}{\|X(x)\|}.$$
Which implies that
$$d(x,P_x(y))\leq \length(l)\leq C^5\lambda_0\frac{\|X(x_1)\|}{\|X(x)\|}d_{x_1}^*(x_1,P_{x_1}(y_1)).$$
Suppose this claim for a while, as a corollary of~\eqref{eq:compatable}, we conclude the proof of this lemma:
\begin{equation*}
\begin{split}
d^*_x(x,P_x(y))& <Cd(x,P_x(y))\\
& <C^6\lambda_0\frac{\|X(x_1)\|}{\|X(x)\|}d_{x_1}^*(x_1,P_{x_1}(y_1))\\
& =\frac{1}{b_0}\frac{\|X(x_1)\|}{\|X(x)\|}d_{x_1}^*(x_1,P_{x_1}(y_1)).\\
\end{split}
\end{equation*}
It remains to prove this claim.

For any $\tilde{z}_1\in \tilde{l}_1$, denote $\tilde{v}$ a tangent vector of $\tilde{l}_1$, write
$\tilde{z}_0=\phi_{-1}(\tilde{z}_1)$, $z=H(\tilde{z})\in l$ and $v=dH^{-1}(\tilde{v})$. 

Then by sectional-hyperbolic, 
$$\cA(\Phi_{-1}(\tilde{v}),X(\tilde{z}_0))\leq \lambda_0 \cA(\tilde{v},X(\tilde{z}_1)).$$

Because $|t_x|_{\tilde{l}_0}|<t_0$, by the assumption on $t_0$, 
\begin{equation}\label{eq.2}
\cA(\Phi_{t_x(\tilde{z}_0)}\cP[\Phi_{-1}(\tilde{v}),X(\tilde{z}_0)]) \leq C\lambda_0 \cA(\tilde{v},X(\tilde{z}_1)).
\end{equation}
Since $\Phi_{t_x(\tilde{z}_0)}(\Phi_{-1}(\tilde{v}))=\Phi_{-1+t_x(\tilde{z}_0)}(\tilde{v})$ and
$\Phi_{t_x(\tilde{z}_0)}(X(\tilde{z}_0))=X(z)$,  we have
\begin{equation}\label{eq.3}
\cA(\Phi_{t_x(\tilde{z}_0)}\cP[\Phi_{-1}(\tilde{v}),X(\tilde{z}_0)])=\cA(\Phi_{-1+t_x(\tilde{z}_0)}(\tilde{v}),X(z)).
\end{equation}

Note that $dH(\tilde{v})$ is the projection of $\Phi_{-1+t_x(\tilde{z}_0)}(\tilde{v})$ along $X(z)$ on
$T_zB^\bot(x)$, combine equations~\eqref{eq.2},~\eqref{eq.3}, we have that
\begin{equation}
\cA(v,X(z))= \cA(\Phi_{-1+t_x(\tilde{z}_0)}(\tilde{v}),X(z)) \leq C\lambda_0 \cA(\tilde{v},X(\tilde{z}_1)).
\end{equation}

By the assumptions (a) and (b) above, $C^{-2}\|v\|\|X(x)\|\leq C^3 \lambda_0\|\tilde{v}\|\|X(x_1)\|$, which implies:
$$\|dH_{\tilde{z}}\|= \frac{\|v\|}{\|\tilde{v}\|}\leq C^5\lambda_0\frac{\|X(x_1)\|}{\|X(x)\|}.$$
\end{proof}

\begin{lemma}\label{l.measureawayfromsingularity}

For every $\vep>0$, there are $\delta>0$ and $K>0$ such that for any non-trivial ergodic measure $\mu$ which
satisfies $\mu(B_\vep^h(\sigma_i))=0$ for every $i=1,\dots, k$, and for any $x\in\supp(\mu)$, $B_\infty(x,\delta)$
is a flow segment with length bounded by $K>0$.

\end{lemma}

\begin{proof}
By Lemma~\ref{l.awayfromIi} and the assumption, $\supp(\mu)$ is away from singularities.
Apply Lemma~\ref{l.localholonomy} on $\{x_i; i\in \mathbb{N}\}$, there is $\delta>0$
such that for any $y\in B_\infty(x,\delta)$:
$$d_{x_n}^*(x_n,P_{x_n}(y_n))>b^n_0\frac{|X(x_n)|}{|X(x)|}d^*_x(x,P_x(y)).$$

Note that $\supp(\mu)\cap \Sing(X)=\emptyset$, $X|_{\supp(\mu)}$ is uniformly bounded from above and below.
Because $d_{x_n}^*(x_n,P_{x_n}(y_n))$ is bounded from above, we conclude that 
$d^*_x(x,P_x(y))=0$, i.e., $y$ belongs to the local orbit of $x$, and
there is $K>0$ such that each connected component of $\Orb(x)\cap B_\delta(x)$ has length bounded by $K$.
We complete the proof.
\end{proof}

Now let us deal with the case $\supp(\mu)$ contains singularities. We will show that there exists 
$\delta>0$ such that for every $y\in B_\infty(x,\delta)$, $d_{x_n}^*(x_n,P_{x_n}(y_n))$ is expanding when $x$ and $x_n$ both 
are away from singularities. 

\begin{definition}\label{d.accumulated}

An extreme singularity $\sigma_i$ ($1\leq i \leq k$) is {\it accumulated} if there is a non-trivial
ergodic measure $\mu$ of $\phi_t$, such that $\mu(B^h_{\delta_0}(\sigma_i))>0$. 
\end{definition}

\begin{remark}\label{r.accumulated}
By Lemma~\ref{l.extremepoints}, for each accumulated singularity $\sigma_i$, $B^h_{\delta_0}(\sigma_i)$ 
is topologically hyperbolic. 
\end{remark}

For $1\leq i \leq k$ and $0<\delta\ll \delta_0$, we consider the set 
$$J^h_\delta(\sigma_i)=\{y:y\in B_{\delta_0}^h(\sigma_i)\;\;\; \text{with}\;\;\; \|X(y)\|\leq\delta\}.$$
There is $\delta_1>0$ small enough, such that $\partial(J_{\delta_1}^h(\sigma_i))\cap \partial (B^h_{\delta_0}(\sigma_i))=\emptyset$ where $\partial(.)$ denotes the boundary of a set. 

\begin{lemma}\label{l.nearsingularity}
For every $L>1$, there are $0<\vep\ll \delta_1$ and $\delta>0$ such that for each accumulated singularity
$\sigma$ and any $x\in B^h_{\delta_0}
(\sigma)\cap \tilde{\Lambda}$ which satisfies 
\begin{itemize}

\item $(x_t)_{t\in[0,T]}\subset B^h_{\delta_0}(\sigma)$ and $(x_t)_{t\in[0,T]}\cap B^h_{\vep}(\sigma)\neq \emptyset$.

\item $x,x_{T}\notin J_{\delta_1}(\sigma)$, $(x_t)_{t\in[0,1]} \cap \partial{J^h_{\delta_1}}(\sigma)\neq 
      \emptyset$ and $(x_t)_{t\in[T-1,T]} \cap \partial{J^h_{\delta_1}}(\sigma)\neq\emptyset$.

\end{itemize}

then for any $y\in B_\infty(x,\delta)$, we have
$$d^*_{x_{T}}(x_{T},P_{x_{T}}(y_{T}))>Ld^*_x(x,P_x(y)).$$

\end{lemma}

In order to simplify the proof, we may assume $T\in \mathbb{N}$. The following proof depends on a local product structure in $\mathcal{\cF}^{cu}_{x_n}(x_n)$ ($n=0,\dots, T$), which blends two 
different families of fake foliations (see Remark~\ref{r.flowfakefoliation}): 
\begin{itemize}
\item $\hat{\cF}_\sigma^i$ ($i=s,c,u,cs,cu$) which comes from
the partially hyperbolic splitting $E^s\oplus E^c\oplus E^u$ of the singularities;

\item $\tilde{\cF}^i_x$ ($i=s,cu$) which comes from the sectional-hyperbolic splitting $E^s\oplus E^{cu}$ defined on
      $\tilde{\Lambda}$.
\end{itemize}

\begin{proof}[Proof of Lemma~\ref{l.nearsingularity}:]

First we define a new 1-dimensional center fake foliation $\overline{\cF}^c_{x_t}$ in $\tilde{\cF}^{cu}_{x_t}(x_t)$, 
$0\leq t\leq T$ which is induced by $\hat{\cF}^{cs}_{x_t}$:
$$\overline{\cF}^c_{x_t}(z)=\hat{\cF}^{cs}_{x_t}(z)\cap \tilde{\cF}^{cu}_{x_t}(x_t),\;\; \text{for every} \;\; z\in \cF^{cu}_{x_t}(x_t).$$
By the local invariance of the fake foliations $\tilde{\cF}^{cu}_{x_t}$ and $\hat{\cF}^{cs}_{x_t}$, the 
new center fake foliation is also locally invariant, i.e. 
$$f(\overline{\cF}^c_{x_t,r_0}(z))\subset \overline{\cF}^c_{x_{t+1}}(z_1)\;\;\; \text{for any}\;\;\; z\in \tilde{\cF}^{cu}_{x,r_0}\;\; \text{and}\;\; t\in[0,T-1].$$

Let $\overline{\cD}^u_x\ni x \subset \tilde{\cF}^{cu}_x(x)$ be a leaf with dimension $dim(E^u_\sigma)$ and tangent to the
$a_0$-$E^{u}$ cone (which is defined near to $\sigma$). By the forward invariance of the $E^u$ cone, this leaf is 
expanding for $f^n$ $0\leq n \leq T$ and the connected component of $f^n(\overline{\cD}^u_x)\cap \tilde{\cF}^{cu}_{x_n}(x_n)$
which contains $x_n$, denoted by $\overline{\cD}^u_{x_n}$, is still tangent to the $E^u$ cone. Hence, 
there is a unique transverse intersection between $\overline{\cD}^u_{x_n}$ and $\overline{\cF}^{c}_{x_n}(z)$ for any $z\in \tilde{\cF}^{cu}_{x_n}(x_n,r_0)$, we denote this intersection by 
$$[z,x_n]=\overline{\cF}^{c}_{x_n}(z)\pitchfork\overline{\cD}^u_{x_n}.$$

This local product structure is preserved by the iteration of $f$. In fact, by the local invariance of 
fake foliation $\overline{\cF}^{c}_{x_n}(.)$ and the expanding property of $\overline{\cD}^u_{x_n}$,
for every $y\in B_{\infty}(x, r_0)$, $f^n([y,x])=[y_n,x_{n}]$ for $0\leq n \leq T$.

Suppose $\delta$ is sufficiently small, there is $r_\delta>0$ such that $P_z|_{B_\delta(z)}$ is well defined and $|t_z|_{B_{\delta}(z)}|\ll r_\delta\ll r_0$ for every $z\in B_{\delta_0}^h(\sigma)\setminus J^h_{\delta_1}(\sigma)$. The choice of $\delta$ depends on $\vep$ and will be given later. Now consider $y\in B_\infty(x,\delta)$ and write
$\tilde{y}_T=P_{x_T}(y_T)$. Then $\tilde{y}=f^{-T}(\tilde{y}_T)=y_t$ for some $|t|<r_{\delta}$. In particular, 
by Corollary~\ref{c.localcoherence}, $\tilde{y}\in \tilde{\cF}^{cu}_{x}(x)\cap B_\infty(x,\delta+r_{\delta} D)$, recall that $D=\max \|X(.)\|$.

Denote $\tilde{z}=[\tilde{y},x]$, by the invariance of the product structure, $\tilde{z}_n=[\tilde{y}_n,x_n]$
for any $0\leq n \leq T$. Take $\tilde{l}_{T}=\tilde{l}_{T}^c\cup \tilde{l}_{T}^u$ a piecewise smooth curve
connecting $\tilde{y}_T=P_{x_T}(y_T)$ and $x_T$ induced from the local product structure: 
$\tilde{l}_{T}=\tilde{l}_{T}^c\cup \tilde{l}_{T}^u$ where 
$\tilde{l}_{T}^c\subset \overline{\cF}^c_{x_T}(\tilde{y}_T)$ connects 
$\tilde{y}_T$ and $\tilde{z}_T$; and $\tilde{l}_{T}^u$ is a shortest segment contained in $\overline{\cF}^u_{x_T}(x_T)$ which connects 
$x_T$ and $\tilde{z}_T$. For $0\leq n \leq T$, write $\tilde{l}_{n}=\tilde{l}^c_{n}\cup \tilde{l}^u_{n}$ a 
piecewise smooth curve connects $x_n$ and $\tilde{y}_n$ where $\tilde{l}_{n}^c=f^{n-T}(\tilde{l}_{T}^c)$ and 
$\tilde{l}_{n}^u=f^{n-T}(\tilde{l}_{T}^u)$. Then $\tilde{l}_{n}^c$ is a fake center curve contained in 
$\overline{\cF}^c_{x_n}(\tilde{y}_n)$ which connects $\tilde{y}_n$ and $\tilde{z}_n$. Recall that $\overline{\cD}^u_{x_T}$ 
is exponentially contracting under the map $f^{n-T}$, $\tilde{l}^n_u$ is contained
in the fake unstable leaf $\overline{\cD}^u_{x_n}$. By the uniform transversality between these two fake leaves, 
there are $C_1>0$ and $0<\lambda_0<1$ such that
\begin{equation}\label{e.cu}
\begin{split}
& \max\{\length(\tilde{l}^u_{(T)}),\length(\tilde{l}^c_{(T)})\} \leq C_1d(x_T,\tilde{y}_T), \\ 
& \length(\tilde{l}^c_{(n)})\leq C_1(\delta+r_\delta D)\;\; \text{and}\;\; \length(\tilde{l}^u_{n})\leq \lambda_0^{n-T}\length(\tilde{l}^u_T). \\
\end{split}
\end{equation}

Hence $\length(\tilde{l}_{(n)})\leq 2C_1(\delta+r_\delta D)$. Now we give the condition $\delta$ need to satisfy. Let $\delta$ be sufficiently small such that:
\begin{itemize}

\item $P_x(\tilde{l}_{0})=
l$ is well defined and $|t_x|_{\tilde{l}_{0}}|\leq t_0\ll r_0/4$, where $t_0$ satisfies that 
for any 2-dimensional subspace $\Sigma$ contained in the tangent space and for any $|t|\leq t_0$, $|\Jac(\Phi_t|_{\Sigma})|\leq C_1$.

\item by Lemma~\ref{l.distancecompactable}:
$$d(x,P_x(y))\leq d^*_x(x,P_x(y))\leq C_1d(x,P_x(y))$$ 
$$d(x_T,\tilde{y}_T)\leq d^*_{x_T}(x_T,\tilde{y}_T)\leq C_1d(x_T,\tilde{y}_T).$$

\end{itemize}

We claim that there is $C_2$ such that $\frac{1}{C_2}<\frac{\length(\tilde{l}^u_0)}{\length(l^u)}<C_2$,
which follows from the observation that the half neighborhood $B^h_{\delta_0}(\sigma)$ is topologically hyperbolic,
i.e., if we denote $\hat{\cF}^{cs,h}_\sigma(\sigma, r_2)$ the corresponding branch of center stable fake leaf located
in the same side as the half neighborhood, then by Lemma~\ref{l.extremepoints}, $\hat{\cF}^{cs,h}_\sigma(\sigma, r_2)$ is contained in the stable
set of $\sigma$. When $\vep$ is sufficiently small, $x$ is close to $\hat{\cF}^{cs,h}_\sigma(\sigma, r_2)$, therefore, the flow
vector $X(x)$ is tangent to the $E^s\oplus E^c$ cone, and is transverse to the $E^u$ cone. Note that $\tilde{l}^u_0$ is 
tangent to the $E^u$ cone, it has a uniform angle with the flow direction. Then one may apply the idea of the proof
in Lemma~\ref{l.localholonomy}. More precisely, consider the holonomy map between $\tilde{l}^u_0$ and
$l^u$ induced by the flow, we are going to 
compare the areas of parallelograms generated by the flow directions and the tangent vectors. Because the Jacobian restricted
on the parallelograms are bounded, and which are equivalent to the product of length of vector of flow and the tangent vector,
by the reason of the uniform angle between them. We conclude the proof of this claim.

By the claim above, we deduce that 
\begin{equation}\label{e.u}
\length(l^u)\leq C_2\lambda_0^T\length(\tilde{l}^u_T).
\end{equation}
 
The estimation of $\length(l^c)$ is quite different, we use a similar argument as in the proof of 
Lemma~\ref{l.localholonomy}. Let us give more explanation. Consider the tangent map of the 
holonomy map $H: \tilde{l}^c_T\rightarrow l^c$ induced by the flow, 
which contracts exponentially the area of the parallelogram formed by the tangent vector of 
$\tilde{l}^c_T$ and the flow vector. For this parallelogram and its image, both areas are equivalent to the product of the length of the tangent vector and the length of the flow vector, we conclude the exponentially contraction of the map $H$.
There is $C_3>0$ such that that
\begin{equation}\label{e.c}
\length(l^c)\leq \length(\tilde{l}^c_T) \lambda_0^{T}C_3^3\frac{\|X(x_T)\|}{\|X(x)\|}.
\end{equation}

Let $\hat{C}=\max\{C_1,C_2, C_3\}$. We now give the condition what $\vep$ need to satisfy. It is easy to see that 
$\lim_{\vep\to 0}T\to \infty.$ 
We assume $\vep$ is sufficiently small such that $\lambda_0^T\hat{C}^3(\hat{C}^2D/\delta_1+1)\leq L^{-1}$, 
that is, $T>-\frac{\ln L\hat{C}^3(\hat{C}^2D/\delta_1+1)}{\ln \lambda_0}$.

Because $x\notin J_{\delta_1}^h(\sigma)$, $\|X(x)\|\geq \delta_1$. By \eqref{e.cu}, \eqref{e.u} and \eqref{e.c}, one has that 
\begin{equation*}
\begin{split}
 d^*_x(x,P_x(y)) & \leq \hat{C} d(x,P_x(y)) \leq \hat{C}(\length(l^c)+\length(l^u))\\
& \leq \hat{C}[C_2\lambda_0^T\length(\tilde{l}^u_T)+\lambda_0^{T}C_3^3\frac{\|X(x_T)\|}{\|X(x)\|}\length(\tilde{l}^c_T)]\\ 
& \leq  \lambda_0^{T}\hat{C}^3(\hat{C}^2\frac{D}{\delta_1}+1)d_{x_T}^*(x_T,P_{x_T}(y_T)) \\
& \leq \frac{1}{L}d_{x_T}^*(x_T,P_{x_T}(y_T)). \\
\end{split}
\end{equation*}

The proof is finished.
\end{proof}

\subsubsection{Proof of Proposition~\ref{p.lorenzexpansive}\label{sss.entropyexpansive}}

Now we are preparing to finish the proof of Proposition~\ref{p.lorenzexpansive}.

\begin{proof}[Proof of Proposition~\ref{p.lorenzexpansive}:]

We first prove that $\phi_t$ is entropy expansive restricted on a small attracting neighborhood of $\Lambda$.

By Lemma~\ref{l.entropyexpansive}, we need show that there is $\delta>0$ such that 
for any invariant ergodic measure $\mu$ of $f$ and for $\mu$ almost every $x\in M$, we have 
$ h_{top}(f, B_\infty(x,\delta))=0.$

Fix $1<b_0<\frac{1}{\lambda_0}$. Let 
$$D^{'}=\min\{\frac{\|X(x)\|}{\|X(x^{'})\|};\text{there are singularities}\;\; \sigma, \sigma^{'}\;\text{such that}$$
$$x\in (\phi_t(\partial(J^h_{\delta_1}(\sigma))))_{t\in [0,1]} \;\; \text{and}\;\; x^{'}\in (\phi_t (\partial (J^h_{\delta_1}(\sigma^{'}))))_{t\in [-1,0]} \}.$$
For $L=b_0/D^{'}$,
apply Lemma~\ref{l.nearsingularity}, we obtain $\vep$ and $\delta$. By taking a smaller value, we assume that 
$\delta<r_0/2$ and satisfies Lemma~\ref{l.localholonomy} for the $\vep$ above.

If $\mu$ is trivial, that means, it is
an atomic measure supported on a singularity $\sigma$. Then by Lemma~\ref{l.singularitiesexpansive},
$B_\infty(x,\delta)$ is a 1-dimensional center segment with length bounded by $L_0>0$, hence has 
vanishing topological entropy.

From now on, we suppose $\mu$ is non-trivial, and $x$ is any generic point of $\mu$. By 
Lemma~\ref{l.singularitiesexpansive}, we assume that $x\notin B_{r_0}(\Sing(X))$. 

Define a sequence of integers 
$0\leq T_1<T_1^{'}<T_2<T_2^{'}\dots$ such that for each $n\geq 1$, the following conditions are satisfied:
\begin{itemize} 
\item $(x_t)_{[T^{'}_n, T_{n+1}]}\cap B_{\vep}(\Sing(X))=\emptyset$;

\item there is a singularity $\sigma$ such that $(x_t)_{[T_n,T_n^{'}]}\subset B^h_{r_0}(\sigma)$ and $(x_t)_{[T_n,T_n^{'}]}\cap B^h_{\vep}(\sigma)\neq \emptyset$; 
\item $x_{T_n},x_{T_n^{'}} \notin J^h_{\delta_1}(\sigma)$, $(x_{t})_{[T_n,T_n+1]}\cap \partial{J^h_{\delta_1}(\sigma)}\neq \emptyset$
 and $(x_{t})_{[T_n^{'}-1.T_n^{'}]}\cap \partial{J^h_{\delta_1}(\sigma)}\neq \emptyset$.

\end{itemize}

If the above sequence is bounded, this implies that $\mu(B_\vep(\sigma))=0$. Then by 
Lemma~\ref{l.measureawayfromsingularity}, $B_\infty(x,\delta)$ is a 1-dimensional segment with bounded length, 
therefore has topological entropy zero.

Hence we assume the sequence is defined with infinite length. For $n\geq 1$, by Lemma~\ref{l.localholonomy}, we have
\begin{equation*}
\begin{split}
d^*_{x_{T_{n+1}}}(x_{T_{n+1}},P_{x_{T_{n+1}}}(y_{T_{n+1}})) & > b_0^{T_{n+1}-T_{n}^{'}}\frac{\|X(x_{T_{n+1}})\|}{\|X(x_{T_n^{'}})\|}d^*_{x_{T_{n}^{'}}}(x_{T_{n}^{'}},P_{x_{T_{n}^{'}}}(y_{T_{n}^{'}})) \\
& \geq b_0^{T_{n+1}-T_{n}^{'}}D^{'}d^*_{x_{T_{n}^{'}}}(x_{T_{n}^{'}},P_{x_{T_{n}^{'}}}(y_{T_{n}^{'}})). \\
\end{split}
\end{equation*}
And by Lemma~\ref{l.nearsingularity}, 
$$d^*_{x_{T_{n}^{'}}}(x_{T_{n}^{'}},P_{x_{T_{n}^{'}}}(y_{T_{n}^{'}}))>\frac{b_0}{D^{'}}d^*_{x_{T_{n}}}(x_{T_{n}},P_{x_{T_{n}}}(y_{T_{n}})).$$
As a conclusion, we conclude that 
$$d^*_{x_{T_{n}^{'}}}(x_{T_{n}^{'}},P_{x_{T_{n}^{'}}}(y_{T_{n}^{'}}))>b_0^nd^*_{x_{T_1}}(x_{T_1},P_{x_{T_1}}(y_{T_1})).$$
This implies that $d^*_{x_{T_1}}(x_{T_1},P_{x_{T_1}}(y_{T_1}))=0$. Therefore, $y$ belongs to the local orbit of $x$. In particular,
$B_\infty(x,\delta_1)$ is a segment of the orbit of $x$, and has topological entropy zero.

It remains to prove that when $\Lambda$ is a quasi attractor and all the singularities are hyperbolic, 
this entropy expansiveness also holds for nearby flows. Take $U$ be a sufficiently small contratcing neighborhood.
Because the singularities are all isolated, each singularity can be treated as
extreme points, and $\Lambda$ contains only finitely many singularities. By a similar argument as Lemma~\ref{l.extremepoints}, one can show that each singularity has negative center exponent. One may check that Lemmas~\ref{l.localholonomy},~\ref{l.measureawayfromsingularity} and~\ref{l.nearsingularity}
all hold robustly for nearby flows restricted in $U$, so does the entropy expansiveness.

The proof is complete.
\end{proof}

\section{Lower semi-continuation of topological entropy~\label{s.lowercontinuiation}}
In this section, we will obtain the lower semi-continuation of the topological entropy.
In Subsubsection~\ref{ss.Katok}, we assume Theorem~\ref{t.shadowing}, a $\C^1$ version of Pesin theory
for flows, to finish the proof of Theorem~\ref{t.lowercontinuiation} and Corollary~\ref{c.homoclinicclass}. The proof of Theorem~\ref{t.shadowing} is postponed to Appendix~\ref{s.Pesin}. The proof of
Corollary~\ref{main.B}, Corollary~\ref{main.C} and Theorem~\ref{main.D} are given in Subsubsection\ref{ss.rest}.

\subsection{Proof of Theorem~\ref{t.lowercontinuiation} and Corollary~\ref{c.homoclinicclass}\label{ss.Katok}}

First let us state a $\C^1$ version of Pesin theory for flows, where we replace the regularity of maps by domination
on the tangent bundle. Similar
statements for diffeomorphisms can be found at~\cite{ABC,Y}.

\begin{theorem}\label{t.shadowing}
Let $\Lambda$ be a compact invariant set of flow $\phi_t$ and admit a dominated splitting
$E\oplus F$. Suppose $\mu$ is a non-trivial hyperbolic ergodic measure supported on $\Lambda$ with index $dim(E)$ 
and $\tilde{\mu}$ an ergodic decomposition of $\mu$ respect to $f$. Then there is a compact, positive $\tilde{\mu}$ measure subset 
$\Lambda_0\subset \Lambda\setminus \Sing(X)$ which satisfies the following properties: 
for any $\vep>0$, there are $L,L^{'}, \delta>0$ such that for any $x$ and $f^n(x)\in \Lambda_0$ 
with $n>L^{'}$ and $d(x,f^n(x))<\delta$, there exists a point $p\in M^d$ and a $\C^1$ strictly 
increasing function $\theta:[0,T] \rightarrow \mathbb{R}$ such that:
\begin{itemize}

\item[(a)] $\theta(0)=0$ and $1-\vep <\theta^{'}(t)<1+\vep$;

\item[(b)] $p$ is a periodic point of $\phi_t$ with $\phi_{\theta(T)}(p)=p$;

\item[(c)] $d(\phi_t(x),\phi_{\theta(t)}(p))\leq \vep\|X(\phi_t(x))\|, t\in [0,T]$;

\item[(d)] $d(\phi_t(x),\phi_t(p))\leq L d(x, \phi_T(x))$

\item[(e)] $p$ has uniform size of stable manifold and unstable manifold.

\end{itemize}
\end{theorem}

Now we are ready to give the proof of Theorem~\ref{t.lowercontinuiation}, the argument is similar to~\cite{K80}. 

\begin{proof}[Proof of Theorem~\ref{t.lowercontinuiation}:]
By Theorem~\ref{t.shadowing}, take $\tilde{\mu}$ an ergodic decomposition of $\mu$ respect to $f$ and $\Lambda_0$
the compact set with positive $\tilde{\mu}$ measure, and for $\vep>0$, take $L,L^{'},\delta>0$. Replacing by a subset, 
we may assume the diameter of $\Lambda_0$ is small enough such that any two distinct periodic orbits obtained in
Theorem~\ref{t.shadowing} are homoclinic related to each other.
By Lemma~\ref{l.entropyforflow}, $h_{\mu}(f)=h_{\tilde{\mu}}(f)>0$.

The following is the same as the proof of~\cite{K80}[Theorem 4.3]. For
$\eta,l>0$ and $n\in \mathbb{N}$, there is a finite set $K_n=K(\eta,l)$ which satisfies the following properties 
(for some $n$ the set $K_n(\eta,l)$ may be empty).
\begin{itemize}

\item $K_n\subset \Lambda_0$;

\item if $x,y\in K_n$, then $d^f_n(x,y)>\frac{1}{l}$;

\item For every $x\in K_n$, there exists a number $m(x)$: $n\leq m(x)\leq (1+\eta)n$ such that
$f^{m(x)}(x)\in \Lambda_0$ and $d(x,f^m(x))<\frac{1}{4Ll}$;

\item For every $\eta>0$, $\lim_{l\rightarrow \infty}\lim_{n\rightarrow \infty}\frac{\ln \Card K_n(\eta,l)}{n}\geq h_{\tilde{\mu}}(f)-\eta$.
\end{itemize}

Let $n,l$ be sufficiently large such that $n>L^{'}$, $\frac{1}{4l}<\delta$ and $\Card K_n(\eta,l)>\exp^{n(h_{\tilde{\mu}}-\eta)}$.

Then for every $x\in K_n$, there is a pseudo orbit $(\phi_t(x))_{[0,m(x)]}$, which is shadowed by a periodic point $p_x$
with period at most $n(1+\vep)(1+\eta)$
and $d^f_n(x,p_x)\leq \frac{1}{4l}$. In particular, for $x\neq y \in K_n$, 
$$d^f_n(p_x,p_y)\geq d^f_n(x,y)-d^f_n(x,p_x)-d^f_n(y,p_y)\geq \frac{1}{2l}.$$
When $x$ and $y$ are both shadowed by the same periodic orbit $\cO$, 
write $p_x=\phi_{t_{y,x}}(p_y)$, then $|t_{y,x}|>\frac{1}{2lD}$ where $D=\max \{\|X(.)\|\}$.

As a corollary of the above argument, there are at most $[n(1+\vep)(1+\eta)2lD]$ 
points of $K_n$ are shadowed by the same periodic orbit. Hence, 
there are at least $\frac{e^{n(h_{\tilde{\mu}}-\eta)}}{[n(1+\vep)(1+\eta)2lD]}$ many different periodic points with period bounded by $n(1+\eta)$. Moreover, these periodic orbits are homoclinic related to each other near to $\Lambda_0$, and
form a Horse-shoe with topological entropy larger than or equal to 
$$\ln \frac{\exp^{n(h_{\tilde{\mu}}-\eta)}}{[n(1+\vep)(1+\eta)2lD]}/n(1+\eta),$$ 
which converges to $h_{\tilde{\mu}}(f)=h_\mu(f)$ when $\vep,\eta\to 0$.

The proof is complete.
\end{proof}

\begin{proof}[Proof of Corollary~\ref{c.homoclinicclass}:]

By Theorem~\ref{t.shadowing}, take $\tilde{\mu}$ an ergodic decomposition of $\mu$ respect to $f$ and $\Lambda_0$
the compact set with positive $\tilde{\mu}$ measure, and for $\vep>0$, take $L,L^{'},\delta>0$. Replacing by a subset, 
we may assume the diameter of $\Lambda_0$ is small enough such that for any two distinct periodic orbits obtained in
Theorem~\ref{t.shadowing} are homoclinic related to each other.

Take $x\in \Lambda_0$ a generic point of $\tilde{\mu}$. By the Birkhoff ergodic theorem, there are $0<n_1<\cdots$
such that $f^{n_i}(x)\in \Lambda_0$ for each $i\in \mathbb{N}$ and $\lim_id(x,f^{n_i}(x))\to 0$. By 
Theorem~\ref{t.shadowing}, each pseudo orbit $(x_t)_{[0,n_i]}$ is shadowed by a periodic point $p_{i}$. Moreover, 
these periodic orbits are homoclinic related to each other and $\lim_ip_i\to x$. This implies that $x$, and the
closure of its forward orbit which coincides to $\supp(\tilde{\mu})$, are contained in the homoclinic class. This
implies immediately that the chain
recurrent class of $\Lambda$ contains a non-trivial homoclinic class.

\end{proof}

\subsection{Proof of Corollaries~\ref{main.B},~\ref{main.C} and Theorem~\ref{main.D}\label{ss.rest}}

\begin{proof}[Proof of Corollary~\ref{main.B}]
By Theorem~\ref{main.A} and the variation principle, there is an $f$-invariant measure $\mu$ supported on $\Lambda$
with positive metric entropy. Then Corollary~\ref{main.B} follows from Corollary~\ref{c.homoclinicclass} and 
Theorem~\ref{t.isolated}.
\end{proof}

\begin{proof}[Proof of Corollary~\ref{main.C}:]

Let $\cR_1$ and $\cR_2$ be the residual subsets defined in Theorem~\ref{t.Lorenzlike} and 
Theorem~\ref{t.isolated} respectively. Then for any star flow $\phi_t\in \cR_1\cap \cR_2$, 
every non-trivial Lyapunov stable chain recurrent class of $\phi_t$ is Lorenz-like, and is 
in fact an attractor. In particular, $\phi_t$ has only finitely many attractors.

By the main result of~\cite{MPa02}, there is a $\C^1$ residual subset of flows $\cR_3$ that that for every 
$\phi_t\in \cR_3$, the points converging to the Lyapunov stable chain recurrent calsses of $\phi_t$ is everywhere dense.
Take $\cR=\cR_1\cap \cR_2\cap \cR_3$. Combining the discussion above, for every flow $\phi_t\in \cR$, 
it has only finitely many attractors, the union of the basins of these attractors are open and dense 
in the ambient manifold. We conclude the proof.
\end{proof}

\begin{proof}[Proof of Theorem~\ref{main.D}]

By Proposition~\ref{p.lorenzexpansive}, there are
attracting neighborhood $U$ of $\Lambda$, a $\C^1$ neighborhood $\cU$ of $\phi_t$ and $\delta>0$, such that
the flows in $\cU$ are all $\delta$-entropy expansive. By Lemma~\ref{l.flowrobustranstive}, $h_{top}(.|_{U})$ 
varies upper semi-continuous respect the flows contained in $\cU$.

Choose $U$ sufficiently small, such that for any flow $\phi^{'}_t\in \cU$, its maximal invariant set 
contained in $U$ is sectional-hyperbolic. Then by Theorem~\ref{t.Flowboundary}, 
$h_{top}(\phi_t^{'}|_{U})>0$. Applying Lemma~\ref{l.flowexpansive}, there is a maximal measure $\mu^{'}$ for the map $\phi^{'}|_U$,
which is clearly non-trivial. 
Then $\supp(\mu^{'})$ admits a sectional-hyperbolic splitting $E^s\oplus F^{cu}$, which implies that
$\mu^{'}$ is a hyperbolic measure with index $dim(E^s)$.

Apply Theorem~\ref{t.lowercontinuiation}, for any $\vep>0$, there is a horse-shoe $\Lambda_\vep$ of $Y$ contained in $U$
with topological entropy larger than $h_{top}(\phi_t^{'}|_{U})-\vep$. Because the topological entropy of a horse-shoe varies continuously, 
we conclude the proof of lower semi-continuation
of the topological entropy and complete the proof of Theorem~\ref{main.D}.
\end{proof}

\appendix

\section{Proof of Theorem~\ref{t.shadowing}\label{s.Pesin}}

For $x\in M\setminus \Sing(X)$ and $v\in T_xM$, 
the orthogonal projection of $\Phi_t(v)$ on $\cN_{\phi_t(x)}$ along the flow direction, denoted by $\psi_t: \cN_x\to \cN_{x_t}$, was defined in
Subsection~\ref{ss.ergodictheory}: $$\psi_t(v)=\pi_{\phi_t(x)}(\Phi_t(v))=\Phi_t(v)-\frac{<\Phi_t(v),X(\phi_t(x))>}{\|X(\phi_t(x))\|^2}X(\phi_t(x)).$$
In this section, we need another flow $\psi^*_t: \cN \rightarrow \cN$, which is called {\it scaled linear Poincar\'e flow}:
$$\psi^*_t(v)=\frac{\|X(x)\|}{\|X(\phi_t(x))\|}\psi_t(v)=\frac{\psi_t(v)}{\|\Phi_t|_{<X(x)>}\|},$$

\begin{lemma}\label{l.twococycles}

$\psi^*_t$ is a bounded cocycle on $\cN_\Lambda$ in the following sense: for any $\tau>0$, there is
$C_\tau>0$ such that for any $t\in[-\tau,\tau]$,
$$\|\psi^*_t\|\leq C_\tau.$$

Moreover, let $\mu$ be a non-trivial ergodic measure supported on $\Lambda$, then the two cocycles
$\psi_t$ and $\psi^*_t$ have the same Lyapunov exponents and Oseledets splitting.
\end{lemma}

\begin{proof}

The uniformly boundness comes from the boundness of $\Phi_t$ for $t\in [-\tau,\tau]$.

For $\mu$ almost every $x$, we have $\lim \frac{\log\|\Phi_t(X(x))\|}{t}=0$. This implies that both
$\psi_t$ and $\psi^*_t$ have the same Lyapunov exponents and Oseledets splitting.

\end{proof}

\subsection{Pesin block}

\begin{lemma}\label{l.dominatedsplittinginvariant}

 $X|_{\Lambda\setminus \Sing(X_t)}\subset F$.
\end{lemma}

\begin{proof}

Because $\phi_t|_{\supp(\mu)}$ is transitive (Birkhoff ergodic theorem), by Lemma~\ref{l.bothdominatedsplitting},
we have either $X|_{\Lambda\setminus \Sing(X_t)}\subset E$ or $X|_{\Lambda\setminus \Sing(X_t)}\subset F$.

Since $\mu$ is hyperbolic with index $dim(E)$, by Theorem~\ref{t.exponentsdiffandflow}, $f$ has
$dim(E)$ number of negative exponents and a vanishing center exponent which corresponds to the flow direciton.
The Lyapunov exponents restricted on $E(x)$ is clearly smaller than the exponents on $F(x)$. Hence, $E(x)$ coincides to
the subspace spanned by the sunbundles in the Oseledets splitting corresponding to the negative exponents; and
$F(x)$ is the subspace spanned by the flow direction and the sunbundles in the Oseledets splitting corresponding to 
the positive exponents.
This implies that $X(x)\in F(x)$, and by the transitivity, we conclude that $X|_{\Lambda\setminus \Sing(X_t)}\subset F$.

\end{proof}

\begin{lemma}\label{l.dominatedsplittingonsupport}
For both of $\psi_t$ and $\psi^*_t$, 
$\pi(E)\oplus \pi(F)$ is a dominated splitting on $\cN_{\Lambda\setminus \Sing(X_t)}$,
which is also the Oseledets splitting of $\mu$ corresponding to the negative exponents and positive 
exponents.
\end{lemma}

\begin{proof}

Let us first prove that $\pi(E)\oplus \pi(F)$ is a dominated splitting for $\psi_t$. The proof of $\psi^*_t$ follows
immediately from the relation between $\psi_t$ and $\psi^*_t$.

Consider $\psi_t$, $t<0$. Because $E\oplus F$ is a dominated splitting for the flow, there are 
$C,c_0>0$ such that for any $t<0$, $u$ and $v$ unit vectors contained in $E_x$ and $F_x$ for 
$x\in \Lambda\setminus \Sing(X_t)$, we have
$$\frac{\|\Phi_t(u)\|}{\|\Phi_t(v)\|}>C \exp^{-c_0t}.$$

Because $E\oplus F$ is a dominated splitting for flow, the angles between these two subbundles are uniformly 
bounded from below. By Lemma~\ref{l.dominatedsplittinginvariant}, $X|_{\supp(\mu)\setminus \Sing(X_t)}\subset F$,
for any $u\neq 0\in E$, there is $C^*>0$ such that 
$$\frac{1}{C^*}<\frac{\|\Phi_t(u)\|}{\|\psi_t(u)\|}<\frac{1}{C^*}.$$

From the definition of $\psi_t$, which is a projection, we have $\|\psi_t(v)\|\leq \|\Phi_t(v)\|$. Hence,
$$ \frac{\|\psi_t(u)\|}{\|\psi_t(v)\|}>C\exp^{-c_0t}/C^*,$$
which means that, $\pi(E)\oplus \pi(F)$ is a dominated splitting for $\psi_t$.

Because $\mu$ has index $dim(E)$ and $\pi(E)\oplus \pi(F)$ 
is a dominated splitting for $\psi_t$, by Theorem~\ref{t.exponentsdiffandflow}, we conclude that $\pi(E)\oplus \pi(F)$ is the Oseledets splitting corresponding to the negative and positive Lyapunov exponents of $\mu$ for both $\psi_t$ and $\psi^*_t$.

\end{proof}

\begin{definition}\label{d.quasihyperbolic}
The orbit arc $\phi_{[0,T]}(x)$ is called
$(\lambda,T_0)^*$ quasi hyperbolic with respect to a direct sum splitting $\cN_x = E(x)\oplus F(x)$ and the
scaled linear Poincar\'e flow $\psi^*_t$ if there exists $0<\lambda<1$ and a partition
$$0=t_0<t_1<\dots<t_l=T,\;\;\; \text{where}\;\; t_{i+1}-t_i\in [T_0, 2T_0]$$
such that for $k=0,1,\dots, l$, we have
$$\prod_{i=0}^{k-1}\|\psi^*_{t_{i+1}-t_i}|_{\psi_{t_i}(E(x))}\|\leq \lambda^k;\;\;  \prod_{i=k}^{l-1}m(\psi^*_{t_{i+1}-t_i}|_{\psi_{t_i}(F(x))})\geq \lambda^{k-l},$$
and
$$\frac{\|\psi^*_{t_{k+1}-t_k}|_{\psi_{t_k}(E(x))}\|}{m(\psi^*_{t_{k+1}-t_k}|_{\psi_{t_k}(F(x))})}\leq \lambda^2.$$

\end{definition}

Now we state a $\C^1$ version of Pesin block.

\begin{lemma}\label{l.goodset}
There are $L^{'}, \eta, T_0>0$ and a positive $\tilde{\mu}_1$ measure subset $\Lambda_0\subset \Lambda\setminus \Sing(X)$, such that for any $x, f^n(x)\in \Lambda_0$ with $n>L^{'}$, $\phi_{[0,n]}(x)$ is $(\eta,T_0)^*$ quasi hyperbolic with respect to a direct sum splitting $\cN_x = E(x)\oplus F(x)$ and the scaled linear Poincar\'e flow $\psi^*_t$.
\end{lemma}

\begin{proof}

By Lemma~\ref{l.dominatedsplittingonsupport} and Theorem~\ref{t.exponentsdiffandflow},
on $\tilde{\mu}$ almost every $x$, $\pi(E(x))\oplus \pi(F(x))$ is the Oseledets splitting of $f$ 
for $\psi^*_t$ corresponding to the negative and positive exponents.
Then by subadditive ergodic theorem, 
there is $a<0$ such that
$$\lim_{t\to \infty}\frac{1}{t}\int \log \|\psi_t^*|_{E}\|d\tilde{\mu}<a \;\; \text{and}\;\; \lim_{t\to -\infty}\frac{1}{|t|}\int \log \|\psi_t^*|_{F}\|d\tilde{\mu}<a.$$

Taking $N_0$ sufficiently large, we have 
$$\frac{1}{N_0}\int \log \|\psi_{N_0}^*|_{E}\|d\tilde{\mu}<a \;\; \text{and}\;\; \frac{1}{N_0}\int \log \|\psi_{-N_0}^*|_{F}\|d\tilde{\mu}<a.$$

Consider the ergodic decomposition of $\tilde{\mu}$ for $f^{N_0}$:
$$\tilde{\mu}=\frac{1}{k_0}\big(\tilde{\mu}_1+\dots+\tilde{\mu}_{k_0}\big).$$
By changing the order, we may assume that:
$$\frac{1}{N_0}\int \log \|\psi_{N_0}^*|_{E}\|d\tilde{\mu}_1<a \;\; \text{and}\;\; \frac{1}{N_0}\int \log \|\psi_{-N_0}^*|_{F}\|d\tilde{\mu}_1<a.$$

Applying Birkhoff ergodic theorem on $f^{N_0}$, for $\tilde{\mu}_1$ almost every $x$, 
$$\lim \frac{1}{mN_0}\sum_{i=0}^{m-1}\log\|\psi_{N_0}^*|_{E(f^{iN_0}(x))}\|<a$$ 
and
$$\lim \frac{1}{mN_0}\sum_{i=0}^{m-1}\log\|\psi_{-N_0}^*|_{F(f^{-iN_0}(x))}\|<a.$$
There is $n_x>0$ such that for any $m>n_x$, 
$$\frac{1}{mN_0}\sum_{i=0}^{m-1}\log\|\psi_{N_0}^*|_{E(f^{iN_0}(x))}\|<a\;\; \text{and}\;\; \frac{1}{mN_0}\sum_{i=0}^{m-1}\log\|\psi_{-N_0}^*|_{F(f^{-iN_0}(x))}\|<a.$$

Choose $N_1$ such that the set $\Lambda^{'}=\{x; n(x)<N_1\}$ has positive $\tilde{\mu}_1$ measure. Let 
$\Lambda_0\subset \Lambda^{'}\setminus \Sing(X)$ be compact and has positive $\tilde{\mu}_1$ measure. It follows
immediately that $\tilde{\mu}(\Lambda_0)>0$. By Lemma~\ref{l.twococycles}, we may define
$$K=\max_{|t|\leq N_0,y\in \Lambda\setminus \Sing(X)}\{\sup\{\psi_t^*|_{E(y)}\}; \sup\{\psi_t^*|_{F(y)}\}\}.$$
Choose $N_2$ sufficiently large and $b<0$ such that 
$$\frac{N_2+N_0}{N_0}a+3K<b<0.$$

We claim that for any sequence $n_1<n_2\cdots<n_l$ satisfying $N_2\leq n_{i+1}-n_i\leq N_2+N_0$ for $0\leq i \leq l-1$: 
$$\frac{1}{l}\sum_{i=0}^{l-1}\log \|\psi^*_{n_{i+1}-n_i}|_{E(f^{n_i}(x))}\|<b.$$
A similar formula also holds for $F$ bundle. Let $L^{'}>0$ be sufficiently large, such that for any $n>L^{'}$, 
there always exists a sequence $n_1<n_2\cdots<n_l$ satisfying $N_2\leq n_{i+1}-n_i\leq N_2+N_0$ for each $0\leq i \leq l-1$.
Then by the above claim, we conclude the proof of this lemma.

It remains to prove the claim.
For each $0\leq i<l$, denote 
$$k_i=[\frac{n_{i+1}}{N_0}]-[\frac{n_{i}}{N_0}]-1, n^{'}_i=([\frac{n_i}{N_0}]+1)N_0\;\; \text{and}\;\; n^*_{i+1}=[\frac{n_{i+1}}{N_0}]N_0.$$

Then $n^*_{i}\leq n_i \leq n^{'}_{i}$, $n^*_{i+1}-n^{'}_i=k_i N_0$ and
$$\psi^*_{n_{i+1}-n_{i}}|_{E(f^{n_i}(x))}= \psi^*_{n_{i+1}-n^*_{i+1}}|_{E(f^{n^*_{i+1}}(x))}\circ \psi^*_{k_iN_0}|_{E(f^{n_i^{'}}(x))}\circ\psi^*_{n^{'}_i-n_i}|_{E(f^{n_i}(x))}$$
Observe that $n_{i+1}-n^*_{i+1}\leq N_0$ and $n^{'}_i-n_i\leq N_0$, which imply in particular that 
\begin{equation*}
\begin{split}
\log\|\psi^*_{n_{i+1}-n_i}|_{E(f^{n_i}(x))}\|& \leq 2K+\log \|\psi^*_{k_iN_0}|_{E(f^{n_i^{'}}(x))}\| \\
& \leq 3K+ \log\|\psi^*_{n^{'}_{i+1}-n^{'}_i}|_{E(f^{n_i^{'}}(x))}\|.\\
\end{split}
\end{equation*}

Because $n^{'}_0=0$, we have that
\begin{equation*}
\begin{split}
\frac{1}{l}\sum_{i=0}^{l-1}\log\|\psi^*_{n_{i+1}-n_i}|_{E(f^{n_i}(x))}\|& \leq \frac{1}{l}\sum_{i=0}^{l}\log\|\psi^*_{n^{'}_{i+1}-n^{'}_i}|_{E(f^{n_i^{'}}(x))}\|+3K \\
& \leq \frac{1}{l}\sum_{j=0}^{\frac{n^{'}_{l+1}}{N_0}} \log\|\psi^*_{N_0}|_{E(f^{jN_0}(x))}\|+3K \\
&\leq \frac{n^{'}_{l+1}}{lN_0} a+3K  \leq \frac{N_2+N_0}{N_0} a+3K \leq b\\
\end{split}
\end{equation*}

\end{proof}

\subsection{Liao's shadowing lemma for scaled linear Poincar\'e flow}

In this subsection we will introduce the Liao's shadowing lemma for scaled linear Poincar\'e flow of~\cite{L89}.
Because we need estimate the difference of the time between the pseudo orbit and real orbit: the item 
(d) below, which only follows from the proof, we provide here the idea of the proof:

\begin{theorem}\label{t.liaoclosinglemma}

Given a compact set $\Lambda_0\subset \Lambda\setminus \Sing(X)$, and 
$\eta>0, T_0 > 0$, for any $\vep > 0$ there exist $\delta> 0$ and $L>0$, such that for any $(\eta,T_0)^*$
quasi hyperbolic orbit arc $\phi_{[0,T]}(x)$ with respect to dominated splitting 
$\cN_x = E(x)\oplus F(x)$ and the scaled linear Poincar\'e flow $\psi_t^*$ which satisfies 
$x, \phi_T(x)\in \Lambda_0$ and $d(x,\phi_T(x))\leq \delta$,
there exists a point $p\in M^d$ and a $\C^1$ strictly increasing function $\theta:[0,T] \rightarrow \mathbb{R}$ 
such that
\begin{itemize}

\item[(a)] $\theta(0)=0$ and $1-\vep <\theta^{'}(t)<1+\vep$;

\item[(b)] $p$ is a periodic point with $\phi_{\theta(T)}(p)=p$;

\item[(c)] $d(\phi_t(x),\phi_{\theta(t)}(p))\leq \vep\|X(\phi_t(x))\|, t\in [0,T]$;

\item[(d)] $d(\phi_t(x),\phi_t(p))\leq L d(x, \phi_T(x))$.

\item[(e)] $p$ has uniform size of stable manifold and unstable manifold.
\end{itemize}
\end{theorem}

\begin{proof}

For simplicity, we assume $M$ is an open set in $R^d$, which implies that for every regular point $x\in M$, $\cN_x$ is a $d-1$-dimensional hyperplane. Because $d(x,\phi_{T}(x))$ is sufficiently small, we may suppose that 
$\phi_{T}(x)\in\cN_x$.

The first step is to translate the problem into the shadowing lemma for a sequence of maps:
By~\cite{GY}[Lemmas 2.2, 2.3], there is $\beta$ depending on $T_0$ such that the holonomy map induced by flow is well defined between
$$\cP_{x,\phi_t(x)}: \cN_x(\beta\|X(x)\|)\rightarrow \cN_{\phi_t(x)}\;\;\; \text{for any}\;\;\; t\in [T_0,2T_0].$$
This is conjugate to the following sequence of maps:
$$\cP^*_{x,\phi_t(x)}, \cN_x(\beta)\rightarrow \cN_{\phi_t(x)}:  \cP^*_{x,\phi_t(x)}(a)=\frac{\cP_{x,\phi_t(x)}(a\|X(x)\|)}{\|X(\phi_t(x))\|}.$$
The tangent map of $\cP^*_{x,\phi_t(x)}$ is uniformly continuous.

Take the sequence of times $0=t_0<t_1<\dots<t_l=T$ in the definition of quasi hyperbolic orbit. 
We consider the local diffeomorphisms induced by the holonomy maps:
$$\cT_i= \cP^*_{\phi_{t_{i-1}}(x),\phi_{t_{i}}(x)}\;\;\; \text{for}\;\;\; 1\leq i\leq l-1;$$
and $ \cT_{l}: \cN_{\phi_{t_{l-1}}(x)}(\beta)\rightarrow \cN_x$. Because $d(x,\phi_T(x))$ is sufficiently small, 
the holonomy map $\cT_l$ is well defined. For each $1\leq i \leq l$, we have
$$T\cT_i(0)=\psi^*_{t_{i}-t_{i-1}}|_{\phi_{t_{i-1}}(x)}.$$
Denote $E_i=\psi_{t_i}(E(x))$ and $F_i=\psi_{t_i}(F(x))$. Then, by the definition of quasi hyperbolic, there is $0<\lambda<1$ such that
such that for $k=1,\dots, l$, we have
$$\prod_{i=1}^{k}\|\cT_i\|\leq \lambda^k;\;\;  \prod_{i=k}^{l}m(\cT_i)\geq \lambda^{k-l-1},$$
and
$$\frac{\|\cT_k|_{E_k}\|}{m(\|\cT_k|_{F_k})}\leq \lambda^2.$$

Now we apply the version of Liao's shadowing lemma on discrete quasi hyperbolic maps (\cite{G}[Theorem 1.1], 
see also~\cite{L79, L85, L96}).
On the $d(x,\phi_{T}(x))$ pseudo orbit
$\{0_x,0_{\phi_{t_1}(x)},\dots, 0_{\phi_{t_{l-1}}(x)}\}$, there is $L>0$ and a 
periodic orbit $p\in \cN_x$ with period $l$ for $\{\cT_0,\cdots, \cT_{l-1}\}$, 
whose orbit $Ld(x,\phi_{T}(x))$-shadows the pseudo orbit.

Using the version of flow tubular theorem~\cite{GY}[Lemmas 2.2], one may prove (a), (b) and (c) above. 
Now let us focus on the proof of (d) and (e). We follow the proof of~\cite{G}. By Lemma 3.1 of~\cite{G}, there is a sequence of positive numbers
$\{c_i\}_{i=1}^{l}$ called {\it well adapted} such that:
\begin{itemize}

\item[(i)] $g_j=\prod_{j=1}^k c_j\leq 1 $, $k=1,\dots, l-1$ and $g_{l}=\prod_{j=1}^{l} c_j= 1 $;

\item[(ii)] denote $\tilde{\cT}_j(a)= g_j^{-1}\cT_j(g_{j-1}a)$, then
$$\|T\tilde{\cT}_j(0)|_{E_j}\|\leq \lambda\;\; \text{and}\;\; m(T\tilde{\cT}_j(0)|_{F_j})\geq \frac{1}{\lambda}.$$

\end{itemize}

It was shown in~\cite{G} (p. 631) that the well adapted sequence is uniformly bounded from above and from zero, and the
sequence of hyperbolic maps $\{\tilde{\cT}_j\}_{j=1}^{l}$ are Lipschtez maps (with uniform Lipschtez constant). 

Denote by 
$$\Psi_k= \cT_{k} \cdots \circ \cT_1 \;\; \text{and}\;\; \tilde{\Psi}_k=\tilde{\cT}_{k}\cdots \circ \tilde{\cT}_1.$$ 
It is easy to see that 
\begin{equation}\label{e.conjugation}
\begin{split}
\Psi_k=g_k \tilde{\Psi}_k \;\; & \text{and}  \;\;\Psi_l=\tilde{\Psi}_l. \\
\end{split}
\end{equation}

Observe that $\{0_{x},0_{\phi_{t_1}}(x),\cdots, 0_{\phi_{t_{l-1}}(x)}\}$ is still a $d(x,\phi_{T}(x))$ pseudo orbit of
sequence of hyperbolic diffeomorphisms $\{\tilde{\cT}_j\}_{j=0}^{l-1}$. By the standard shadowing lemma for hyperbolic
diffeomorphisms (see~\cite{G}[Lemma 2.1]), there is $L>0$ such that the pseudo orbit is $Ld(x,\phi_{T}(x))$-shadowed by periodic orbit $\{\tilde{p},\tilde{p}_1,\dots, \tilde{p}_{l-1}\}$, and this periodic orbit has uniform size
of stable and unstable manifold. By \eqref{e.conjugation}, $\{g_k \tilde{p}_k\}_{k=1}^{l}$ is a pseudo orbit for the 
original sequence of maps $\{\cT_k\}_{k=0}^{l-1}$ and $p_0=\tilde{p}_0$. Hence, $p_0$ has uniform size of stable manifold and
unstable manifold. We conclude (e).

It remains to prove (d). We claim that 
\begin{equation}\label{e.sumdistance}
\begin{split}
\sum_{0\leq i \leq l-1} d(p_i,0_{\phi_{t_i}(x)})& \leq L^{*} d(x,\phi_{t_l}(x)). \\
\end{split}
\end{equation}

By the definition of well adapted sequence, $g_k\leq 1$ for every $1\leq k \leq l$, which implies that
$d(p_k,0_{\phi_{t_k}(x)})\leq d(\tilde{p}_k,0_{\phi_{t_k}(x)})$. Hence, to prove this claim, it suffices to verify
that there is $L^{*}>0$ such that 
$$\sum_{0\leq i \leq l-1} d(\tilde{p}_i,0_{\phi_{t_i}(x)})\leq L^{*} d(x,\phi_{t_l}(x)).$$
This property can be obtained directly from the proof of shadowing lemma of uniformly hyperbolic maps. 
Here we state a proof which depends on the shadowing lemma:

\begin{proof}
We will consider {\it $F$-admissible manifolds}, which are manifolds contained in 
$\cN_{\phi_{t_i}(x)}$ ($0\leq i \leq l$) and satisfy: 
\begin{itemize}
\item tangent to the $F$ cone;
\item has uniform size.
\end{itemize}
Take an $F$ admissible manifold $\cI_0^F\ni 0_x\subset \cN_x$.
Denote $a_0= \cI_0^F\pitchfork W^s_{loc}(\tilde{p}_0)$.
By the hyperbolicity of
$\{\tilde{\cT}_j\}_{j=1}^{l}$, $\tilde{\Psi}_1(\cI_0^F)$ contains an $F$ admissible manifold $\cI_1^F\ni 0_{\phi_{t_1}(x)}$. 
Denote $a_1=\cI^F_1\pitchfork W^s_{loc}(\tilde{p}_1)$, then $a_1=\tilde{\cT}_0(a_0)$.
By induction, we define a sequence of $F$ admissible manifolds $\cI^F_j$, $1\leq j\leq l$ such that 
$\tilde{\Psi}_{j}(\cI^F_0)\supset \cI^F_{j}$. Denote $a_j=\cI^F_j\pitchfork W^s_{loc}(\tilde{p}_j)$, 
then $a_j=\tilde{\Psi}_{j}(a_0)$. Moreover, for $0\leq j \leq l-1$, $0_{\phi_{t_j}(x)}\in \cI^F_{j}$. 

Let $\gamma^E_0\subset W^s_{loc}(\tilde{p}_0)$ which connects $\tilde{p}_0$ and $a_0$ with minimal length, and $\gamma^F_l\subset \cI_l^F$
which connects $a_l$ and $\tilde{\Psi}_l(0_x)$ with minimal length. By the uniformly transversality between
the $F$ admissible manifolds and local stable manifolds,
there is $K>0$ such that
$$\length(\gamma^E_0)\leq Kd(x,\tilde{p}_0)\leq LKd(x,\phi_{T}(x))$$
and 
$$\length(\gamma^F_l)\leq Kd(\phi_{T}(x),\tilde{p}_0)\leq LKd(x,\phi_{T}(x)).$$
For $1\leq i \leq l$, denote $\cG_i:  \cN_x \to \cN_{\phi_{t_i}}(x)$:
$$\cG_i=   \tilde{\cT}_{i+1}^{-1}\circ \dots \circ \tilde{\cT_{l}}^{-1}.$$
Then $\gamma_i=\tilde{\Psi}_i(\gamma^E_0)\cup \cG_i(\gamma^F_l)$ is a piecewise smooth curve 
connecting $0_{\phi_{t_i}(x)}$ and $\tilde{\Psi}_i(\tilde{p})$. Moreover, $\tilde{\Psi}_i(\gamma^E_0) \subset W^s_{loc}(\tilde{p}_i)$
and $\cG_i(\gamma^F_l)\subset \cI^F_i$. By the hyperbolicity, $\tilde{\Psi}_i|_{\gamma^E_0}$ and 
$\cG_i|_{\gamma^F_l}$ both are exponentially contracting. There is $\lambda<1$ such that 
$$\length(\gamma_i)< \lambda^i KL d(x,\phi_{T}(x))+\lambda^{l-i}KLd(x,\phi_{T}(x)).$$

Hence, there is $L^*$ such that:
\begin{equation*}
\begin{split}
\sum_{0\leq i \leq l-1} d(\tilde{p}_i,0_{\phi_{t_i}(x)})& \leq \sum \length(\gamma_i) \\
& \leq \sum_{0\leq i\leq l-1} (\lambda^i+\lambda^{l-i})KLd(x,\phi_{T}(x))\\
& \leq L^* d(x,\phi_{T}(x)).\\
\end{split}
\end{equation*}

\end{proof}
 
Let us continue the proof.
For any $x\in M\setminus \Sing(X)$ and $t\in [T_0,2T_0]$, one may define the {\it cross time function}
$t_{x,\phi_{t}(x)}: \cN_x(\beta)\rightarrow \mathbb{R}$:
for any $a\in \cN_x(\beta)$, $\phi_{t_{x,\phi_t(x)}(a)}(a\|X(x)\|)\in \cN_{\phi_t(x)}$. The function
$t_{x,\phi_t(x)}$ is smooth by the implicite function theorem, and whose derivative is uniformly continuous
(for example, bounded by $L_2$), which can be deduced from the proof of ~\cite{GY}[Lemmas 2.3].
Then 
$$\sum_{i=0}^{l-1} |(t_{i+1}-t_i)-t_{\phi_{t_i}(x),\phi_{t_{i+1}}(x)}(p_i)|<L^*L_2 d(x,\phi_{T}(x)).$$
Moreover, there is $L_3$ such that for any $y,z\in \cN_x(\beta|X(\phi_i(x))|)$, and $t_i\leq t<t_{i+1}$, we have 
$$d(\phi_s(y),\phi_s(z))<L_3 d(y,z).$$

Now we are prepare to prove (d):

Suppose $t_i\leq t < t_{i+1}$, denote by $T_i=\sum_{j\leq i-1}t_{\phi_{t_j}(x),\phi_{t_j+1}(x)}(p_j)$.
\begin{eqnarray*}
d(\phi_t(x),\phi_t(p))&=&d(\phi_{t-t_i}(\phi_{t_i}(x)),\phi_{t-T_i}(p_i))\\
&\leq & d(\phi_{t-t_i}(\phi_{t_i}(x)),\phi_{t-t_i}(p_i))+d(\phi_{t-t_i}(p_i),\phi_{t-T_i}(p_i))\\
&<&L_3d(\phi_{t_i}(x),p_i)+|t_i-T_i|D\\
&<&	L_3L^{*}d(x,\phi_{T}(x))+L_2L^{*}Dd(x,\phi_{T}(x)).
\end{eqnarray*}
Let $L=L_3L^*+L_2L^{*}D$, the proof is complete.

\end{proof}
\subsection{Proof of Theorem~\ref{t.shadowing}}

It is easy to see that Theorem~\ref{t.shadowing} is a direct corollary of Lemma~\ref{l.goodset}
and Theorem~\ref{t.liaoclosinglemma}.

\section{Proof of Theorem~\ref{t.isolated}\label{apendix.isolation}}

The following property comes directly
from the definition of Lyapunov stable chain recurrent class, and will be used frequently during
the proof.

\begin{lemma}\label{l.lyapunov}

Let $\Lambda$ be a Lyapunov stable chain recurrent class of $\phi_t$. Then for any $x\in \Lambda$,
its unstable set is contained in $\Lambda$.

\end{lemma}

We also need two {\it $\C^1$ generic properties} for flows.

\begin{proposition}\label{p.homoclinic class}

There is a $\C^1$ residual subset $\cR$ of flows, such that for any flow $\phi_t\in \cR$, and for any chain recurrent
class $C$ of $\phi_t$ which contains a periodic point $p$, $C$ coincides to the homoclinic class of $\Orb(p)$. In particular,
$C$ is transitive.

\end{proposition}

\begin{proposition}\label{p.generic}

There is a $\C^1$ residual subset $\cR$ of flows, such that for any flow $\phi_t\in \cR$ and any non-trivial chain recurrent class
$C$ of $\phi_t$, suppose $C$ contains a hyperbolic singularity $\sigma$ and a hyperbolic periodic point $p$, further suppose there is a dominated splitting $E^s(\sigma)=E^{ss}(\sigma)\oplus E_1^{c}(\sigma)$ on $\sigma$ where $E^c_1$ is 1-dimensional weak contracting subbundle, then the strong stable manifold $W^{ss}(\sigma)$ of $\sigma$ divides the stable manifold into two branches $W^{s,+}(\sigma)$ and $W^{s,-}(\sigma)$, and if $\cC\cap (W^{s,i}(\sigma)\setminus \sigma)\neq \emptyset$ for $i=+,-$, then $W^{s,i}(\sigma)\cap W^u(p)\neq \emptyset$.
\end{proposition}

Proposition~\ref{p.homoclinic class} is the flow version of similar results of diffeomorphisms stated in \cite{BC03}[Corollary 1.4, Remark 1.10]. Proposition~\ref{p.generic} can be proved by applying the connecting lemma of~\cite{BC03} 
on a branch of stable manifold of the singularity and the unstable manifold of the other periodic orbit.

\begin{proof}[Proof of Theorem~\ref{t.isolated}:]
Let $\cR_2$ be the residual subset of flows which are Kupka-Smale and satisfies Propositions~\ref{p.homoclinic class} and
~\ref{p.generic}. 

By Corollary~\ref{main.B} and Proposition~\ref{p.homoclinic class}, $\Lambda$ coincides to the homoclinic class of
a hyperbolic orbit $\Orb(p)$ and is transitive.
In order to prove $\Lambda$ is an attractor, it remains to show that $\Lambda$ is an attracting set. From
the definition of Lyapunov stable chain recurrent class, it suffices to prove that $\Lambda$ is isolated, i.e.,
it cannot be approximated by other chain recurrent classes.

Let $\{U_n\}$ be the sequence of attracting neighborhoods of $\Lambda$ in the definition of Lorenz-like class. 
Suppose by contradiction that $\Lambda$ is not isolated, i.e., there are chain recurrent classes $\cC_n$ and $y^n\in \cC_n$ 
such that $y^n\rightarrow \Lambda$. We may assume that each $\cC_n$ is contained in an attracting neighborhood 
$U_n$. Thus, $\limsup_{n}\cC_n\subset \Lambda$.

Because $\phi_t\in\cR_2$ is Kupka-Smale, the singularities are isolated. Hence, $\phi_t$ contains at most finitely many
singularities. We suppose $\cup_n\cC_n$ contains no singularities. Because $\cC_n$ also induces the 
singular hyperbolic splitting from $\Lambda$, it contains no singularity implies is hyperbolic. 
Therefore $C_n$ contains a hyperbolic periodic orbit $\Orb(x^n)$. Denote $\Lambda_0\subset \Lambda$ 
the Hausdorff limit of $\Orb(x^n)$, which is a compact invariant subset. 

We claim that $\Lambda_0$ contains a singularity. Suppose $\Lambda_0$ contains no singularities. Because it admits a singular hyperbolic splitting, hence, $\Lambda_0$ is a hyperbolic set. For $n$ large enough, the two hyperbolic sets $\cC_n$ and $\Lambda_0$ are homoclinic related to each other. This implies that $\Lambda_0$ and $\cC_n$ are contained in the same chain recurrent class, a contradiction with the assumption that $\Lambda$ and $\cC_n$ are different chain recurrent classes.

Let $\sigma$ be a singularity contained in $\Lambda_0$. Then as in the proof of Lemma~\ref{l.singularities},
$\sigma$ admits a partially hyperbolic splitting $E^{ss}\oplus E^{cs}_1\oplus E^u$ where $E^{cs}_1$ is a 1-dimensional
weak stable bundle and $\cF^{ss}\cap \Lambda=\emptyset$. Then $\cF^{ss}$ divides $W^s(\sigma)$ into two branches $W^{s,+}(\sigma)$ and $W^{s,-}(\sigma)$. 

By modifying $x^n$ to a different point in the same periodic orbit, we assume that $x^n\rightarrow \sigma$. Fix a small
$\vep$, for each $x^n$, let $t_n<0$ be the time satisfying
$$(x^n)_{t_n}\in \partial(B_\vep(\sigma))\;\; \text{and}\;\; (x^n)_{[t_n,0]}\subset B_\vep(\sigma).$$
Denote $z^n=(x^n)_{t_n}$. Replacing by a subsequence, suppose $lim_n z^n = z^0$, then $z^0\in W^s(\sigma)$. 
We may further suppose $y^0\in W^{s,+}(\sigma)$. 

By Proposition~\ref{p.generic}, $W^u(\Orb(p))\pitchfork W^{s,+}(\sigma)\neq \emptyset$. Denote $a$ an intersecting point.
There is $t$ such that $a\in W^{uu}(p_t)$. Consider a disk $D\ni a \subset W^{uu}(p_t)$. 
Then by $\lambda$ lemma, $\phi_t(D) \rightarrow W^u(\sigma)$ from the `right' and $\{\phi_t(D);t>0\}$ is a
submanifold tangent to the $F^{cu}$ bundle with dimension $dim(F^{cu})$ . Therefore for $n$ large enough, $\cF^s(x^n)\pitchfork \{\phi_t(D);t>0\}\neq \emptyset$. Because $\Lambda$ is a Lyapunov stable chain recurrent class, by Lemma~\ref{l.lyapunov}, $\phi_t(D)\subset \Lambda$. 
Thus $\cF^{s}(x^n)\cap \Lambda\neq \emptyset$. In particular, this implies that 
$$d(\phi_t(x^n), \Lambda)=d(\phi_t(x),\phi_t(\Lambda))\to 0.$$ 
Which only occurs when $\cC^n\cap \Lambda\neq \emptyset$, a contradiction.
This completes the proof.
\end{proof}

\end{document}